\documentclass[a4paper,twoside]{article}      
\usepackage{amsmath,amssymb,amsfonts,amsthm,amscd} 
\usepackage{bigints}
\usepackage{graphicx}                 
\usepackage{color}                    
\usepackage{ mathrsfs }
\usepackage{epstopdf}
\usepackage{indentfirst}
\usepackage{enumerate}
\usepackage{url}         
\usepackage{colonequals} 
\usepackage{authblk} 
\usepackage{a4wide}
\newtheorem{theorem}{Theorem}[section]
\newtheorem{proposition}[theorem]{Proposition}

\newtheorem{lemma}[theorem]{Lemma}
\theoremstyle{definition}
\newtheorem{remark}[theorem]{Remark}

\newtheorem{assumption}{Assumption}[section]

\def\!{\mathop{\mathrm{!}}}

\def\R{\mathbb{ R}}
\def\N{\mathbb{ N}}

\def\F{\mathcal{F}}

\newlength{\boxwidth}
\date \today
\title{On assessing the accuracy of defect free energy computations}
\author[1]{Matthew Dobson}
\author[2]{Manh Hong Duong}
\author[2]{Christoph Ortner}
\affil[1]{Department of Mathematics and Statistics, UMass Amherst, 710 N Pleasant Street, Amherst, MA 01003, USA.}
\affil[2]{Mathematics Institute, University of Warwick, Coventry CV4 7AL, UK.} 
\begin{document}
\maketitle

\begin{abstract}
  We develop a rigorous error analysis for coarse-graining of defect-formation
  free energy. For a one-dimensional constrained atomistic system, we establish
  the thermodynamic limit of the defect-formation free energy and obtain
  explicitly the rate of convergence. We then construct a sequence of
  coarse-grained energies with the same rate but significantly reduced
  computational cost. We illustrate our analytical results through explicit computations for the
  case of harmonic potentials and through numerical simulations.
\end{abstract}


\section{Introduction}
Crystalline materials contain a variety of defects, such as vacancies,
interstitials and dislocations. Macroscopic properties of materials are strongly
dependent on the distribution of defects, in particular through the interaction
between dislocations and other defects \cite{callister2010}. Meso-scopic models
for defect interaction (e.g., dislocation dynamics, point defect diffusion)
usually take as input an atomistic simulation of a single, or few defects, from
which the meso-scopic model parameters can be extracted. A prototypical example
is the defect formation energy, which we discuss in more detail below.  A great
number of numerical schemes on spatial coarse-graining of the free energy have
been developed in the literature, see for instance in
\cite{DupuyTadmorMillerPhillips05, Marian2010} and references therein. However,
a rigorous analysis on the accuracy of these schemes is still underdeveloped; we
are only aware of the references \cite{BlancBrisLegollPatz2010,
  LuskinShapeev2014TMP}.

In this paper, we provide such a rigorous analysis for the computations of the
defect-formation free energy. We consider one-dimensional constrained atomistic
systems, which model perfect and defect materials respectively, with degrees of
freedom $u\in\R^N$. The system can be either influenced by external forces or
not. In the case without external forces, free energies
are respectively defined by
\begin{align}
&F_N(A)=-\beta^{-1}\log \int_{\R^{N-1}}\exp\Big[-\beta V(u)\Big]\,du_1\ldots du_{N-1}, \label{F_N(x,0)}
\\&F_N^P(A)=-\beta^{-1}\log \int_{\R^{N-1}}\exp\Big[-\beta V^P(u))\Big]\,du_1\ldots du_{N-1}\label{F_N(x,P)},
\end{align}
where
\begin{equation}
V(u) = \sum_{i=1}^N \psi(u_i-u_{i-1}), \quad V^P(u)=V(u)+P(u)
\end{equation}
are the energies associated to the perfect and defect materials, $V$ is the sum of bond energies $\psi(u_i-u_{i-1})$; $P:R^N\to\R$ models the defect. For simplicity, we assume that $P$ is a localised function and depends only on the first bond $P(u)=P(u_1-u_0)$; and finally $\beta>0$ is the temperature.

In the case with external forces, the perfect free energy is unchanged, but the deformed free energy is influenced by the external forces
\begin{equation}
\label{F_N(x,P) forces}
F_N^P(A)=-\beta^{-1}\log \int_{\R^{N-1}}\exp\Big[-\beta\sum_{i=1}^N \psi_i(u_i-u_{i-1})-\beta P(u_1)\Big]\,du_1\ldots du_{N-1},
\end{equation}
where $\psi_i(y)=\psi(y)+h_i y$ with $\{h_i\}_{i=1}^N$ representing the external forces.  

Note that the integrals \eqref{F_N(x,0)}, \eqref{F_N(x,P)} and \eqref{F_N(x,P) forces} are subjected to the boundary constraints
\begin{equation}
\label{eq: boundary constraint}
u_0=0, \quad u_{N}=NA
\end{equation}
so that the free energies depend on $N$ and $A$ as shown, and $P(u)=P(u_1)$.

The main quantity of interest in this paper is the \textit{defect-formation free
  energy} defined as the difference of the free energies
\begin{equation}
\label{eq: def of GN}
G_N(A):=F_N^P(A)-F_N(A)=-\frac{1}{\beta}\log\frac{\int_{\R^{N-1}}\exp(-\beta V^P(u))\,du}{\int_{\R^{N-1}}\exp(-\beta V(u))\,du}.
\end{equation}
This quantity is used to obtain the equilibrium defect
concentration~\cite{Putnis92, Walsh11} or to analyse defect
clustering~\cite{Seebauer09, Herbert2014}. A direct computation of $G_N[P]$ is
practically impossible due to the curse of dimensionality: one needs to compute
integrals over $\R^{N-1}$, which is an extremely high-dimensional space.

As a matter of fact, $N$ itself is an approximation parameter, the {\em exact}
defect formation free energy is given by the thermodynamic limit, letting
$N \to \infty$. Establishing this limit, and thus making precise what we mean by
the ``exact model'' is the first result of our paper. Once we have established
this, we search for an alternative scheme by which to approximate it, which
yields an improved accuracy/computational cost ratio.

The computation of $\lim_N G_N$ is a problem that is interesting in its own
right, but at the same time it serves as a natural benchmark problem for
exploring the relative accuracy/cost of various coarse-graining methods at
finite termperature.

The work \cite{BlancBrisLegollPatz2010} considers a similar model as ours, but
this work is focused on the scaling limit of the free energy, not the free
energy difference, which is a different scale. Furthermore, it does not take
defects into account. The work \cite{LuskinShapeev2014TMP} is in spirit much
closer to ours and in particular does take defects into account. The main
difference to our work is that \cite{LuskinShapeev2014TMP} consider ``low''
temperature via an asymptotic series expansion. Moreover, our coarse-grained
model has some close similarities with common quasicontinuum-type models.

Technically, to prove our main results, we will link the defect-formation free
energy to a ratio of the densities of certain random variables and employ
techniques from statistical mechanics. The latter have been used in the
literature, for example in \cite{GOVW09, Menz11}. However, the connections to
the defect-formation free energy, to the best of our knowledge, is new and
moreover, some technical modifications of the mentioned papers were required.


\subsection{Assumptions and main results}
For simplicity of notation we set $\beta = 1$ throughout the paper. Moreover, we
make the following standing assumptions on the bond energy $\psi$, the defect
$P$ and the external forces $\{h_i\}_{i=1}^N$.
\begin{assumption}
  \label{ass: assumption}
  $\psi, P \in C^2(\R)$ and there exist positive constants
  $\kappa_1\leq\kappa_2$ and $\varsigma_1\leq\varsigma_2$ such that
  \begin{equation}
    \label{assump on psi and P}
    \kappa_1 \leq \psi''\leq \kappa_2,\quad \varsigma_1\leq (\psi+P)''\leq \varsigma_2.
  \end{equation}
\end{assumption}
\begin{assumption} 
  \label{assum: assumption on asssumption forces}
  $\mathbf{h} := (h_1,h_2 \ldots) \in l^1$; we can then define
  $H := \sum_{i = 2}^\infty h_i$.
\end{assumption}


{\bf Step 1: Thermodynamic limit: } Our first result concerns the {\em rate of
  convergence} of defect formation free energy. Its proof is given in Theorems
\ref{theo: main theorem constrained case} and \ref{theo: thermodynamic limit
  with forces}.

\begin{theorem}
  \label{theo: limit G_N}
  There exists $G_\infty \in C^\infty(\R)$, such that, for all $A\in\R$
  \begin{displaymath}
    | G_N(A) - G_\infty(A) | \lesssim N^{-1}.
  \end{displaymath}
\end{theorem}

{\bf Step 2: Coarse-graining: } The (finite-temperature) Cauchy-Born strain
energy function is given by \cite{BlancLegoll2013}
\begin{equation}
  \label{eq: Cauchy-Born energy}
  W(A) = \sup_{\sigma\in \R}\Big\{\sigma A - 
  \log \int_{\R}\exp(-\beta \psi(y) + \sigma y)\,dy\Big\}.
\end{equation}
Taking a continuum model $\int [W(u') + h u'] dx$ outside the defect core
$\{0, 1\}$ and then discretising it with the atomistic grid
$\{1, 2, \dots N\}$ we obtain
\begin{displaymath}
  E^{\rm cb}_N(u) := \sum_{i = 2}^N \Big[ W(u_i')
  - W(A) + h_i u_i' \Big],
  \qquad \text{where } u_i' = u_i - u_{i-1},
\end{displaymath}
with admissible displacements $u : \{0, \dots, N\} \to \R$ satisfying
$u_0 = 0, u_N = A N$. After replacing $u_i = Ai + v_i$, summation by parts, and
taking the formal limit $N \to \infty$, yields
\begin{displaymath}
  E^{\rm cb}(u) = W'(A) (A-u_1) + A H
  + \sum_{i = 2}^\infty \Big[ W(A+v_i') - W(A) - W'(A) v_i' + h_i v_i' \Big].
\end{displaymath}
It is important to note here that $E^{\rm cb}$ is formulated in a way that
ensures it is well-defined for arguments with $v' \in \ell^2$.

We obtain the following characterisation of $G_\infty(A)$ in terms of
$E^{\rm cb}$.

\begin{theorem}
  \label{th:intro:cg}
  Let $E^{\rm cg}(A, y) := \inf_{u \in \R^\N, u_1 = y} E^{\rm cb}(u)$, then 
  \begin{displaymath}
    G_\infty(A) = - \log \frac{\int_\R 
      \exp\big(  - P(y) - \psi_1(y) - E^{\rm cg}(A, y) \big) 
      dy}{ \int_\R \exp\big( -\psi(y) - E^{\rm cg}_{{\bf h = 0}}(A, y) \big) dy}.
  \end{displaymath}
  where $E^{\rm cg}_{\bf h=0}$ denotes the coarse-grained energy with
  $h_j \equiv 0$.
\end{theorem} 

{\bf Step 3: Approximation: } Thus, we have replaced a limit of high-dimensional
integrals by a one-dimensional integral over a coarse-grained energy functional
whose evaluation requires the solution of an infinite-dimensional variational
problem. In our next step, we replace $E^{\rm cg}(A, y)$ with a
finite-dimensional approximation.

Let 
\begin{equation*}
  E_N^{\rm cg}(A, y) := \inf\limits_{\substack{u\in\R^N\\ u_1=y, u_N=N A}}\,
  E^{\rm cb}_N(u)
\end{equation*}
and
\begin{displaymath}
  G_N^{\rm cg}(A) := 
  - \log \frac{\int_\R \exp\big(  - P(y) - \psi_1(y) - E^{\rm cg}_N(A, y) \big) 
      dy}{ \int_\R \exp\big( -\psi(y) - E^{\rm cg}_{N, {\bf h = 0}}(A, y) \big) dy}.
\end{displaymath}
Here we have chosen $E_N^{\rm cg}$ as the most basic approximation scheme to
$E^{\rm cg}$, but far more sophisticated choices could be explored. With this
definition we obtain the following result.

\begin{theorem}
  \label{th:intro:cg-approx}
  (i) $G_N^{\rm cg}(A)$ is well-defined for all $A \in \R$.
  
  (ii) For all $A \in \R$ we have the estimate
  \begin{displaymath}
    \big| G_N^{\rm cg}(A) - G_\infty(A) \big| \lesssim N^{-1}.
  \end{displaymath}
\end{theorem}

The sharpness of the results of Theorems \ref{theo: limit G_N} and
\ref{th:intro:cg} are demonstrated through explicit computations in the harmonic
case $\psi(y)=\alpha|y|^2$ and $P(y)=\beta |y|^2$ in Section \ref{sec: appendix}
and in numerical simulations in Section \ref{sec: numerical}.

{\bf Interpretation: } Statements (ii) of Theorems \ref{th:intro:cg} and
\ref{th:intro:cg-approx} imply that $G_\infty(A)$ can be computed from two
one-dimensional integrals, but this extreme reduction of computational
complexity is only due to the special one-dimensional structure of our model
problem and cannot in general be reproduced. 

The structure in our construction that can be expected more generally though is
that $G_\infty(A)$ can be {\em approximated by} a low-dimensional canonical
average with respect to a coarse-grained energy that is obtained by a
variational problem in the exterior of the computational domain.
%
%
In our case the coarse-grained measure is one-dimensional but in general one may
still expect it to be relatively low-dimensional. A Langevin or other type of
Markov-Chain type algorithm can now be employed to compute $G_\infty(A)$;
cf. Section \ref{sec: numerical}.

Of course, the evaluation of $E^{\rm cg}(y)$ is in general impossible, and an
approximation needs to be performed. For example, $E_N^{\rm cg}(A, y)$ (and its
derivatives) is computable with a reasonably low $O(N)$ cost. Note that $W$ itself may
be costly to evaluate, but it could be easily precomputed to high accuracy
e.g. via Taylor expansions or spline techniques. The $O(N)$ cost could be
reduced further if we employ a quasi-continuum style coarse-graining of
$E^{\rm cg}_N$. 


%
%
%
%
%
%
%
%
%
%
%
%
%
\subsection{Organisation of the paper}
The rest of the paper is structured as follows. In Section \ref{sec: Noforces}
we study the case without external forces.  Extension to the case
with external forces is shown in Section \ref{sec: forces}. In Section
\ref{sec: appendix}, we provide explicit computations for the harmonic case. Finally, in Section \ref{sec: numerical}, we present some numerical simulations. 
%

\section{The case without external forces}
\label{sec: Noforces}
In this section, we analyse the case without external forces.
\subsection{Thermodynamic limit}
In this section, we prove Theorem \ref{theo: limit G_N} for the case without external forces by establishing the existence of the thermodynamic limit $G_\infty$ and the rate of convergence of $G_N$ to $G_\infty$. The main result of this section is the following theorem.
\begin{theorem} 
\label{theo: main theorem constrained case} 
Suppose that Assumption \ref{ass: assumption} is satisfied.
Then the thermodynamic limit is given by
\begin{equation}
\label{eq: Ginf}
G_\infty(A)=-\log\frac{\int_{\R} \exp[-(\psi+P)(y)+W'(A) y]\,dy}{\int_{\R} \exp[-\psi(y)+W'(A) y]\,dy}.
\end{equation}
Moreover, for all $A\in \R$, we have the estimate
\begin{equation}
|G_N(A)-G_\infty(A)|\lesssim N^{-1}.
\end{equation}
\end{theorem}
\begin{proof}
The proof is split into three steps that are Proposition \ref{prop: difference of energy}, Proposition \ref{prop: limits of difference of G} and Proposition \ref{prop: limit of ratio of densities} below.
\end{proof}
We start with the following auxiliary lemma that links the free energy to the density of an average of independent random variables. This lemma will be applied in Proposition \ref{prop: difference of energy} and Theorem \ref{theo: thermodynamic limit with forces} later on.
\begin{lemma}
\label{lem: aulem 1}
 Suppose that $\tilde{\psi}_i \in C^2(\R)$ and $0 < \kappa_1 \leq \tilde{\psi}_i'' \leq \kappa_2$ for $i=1,\ldots, N$. We define
\begin{align}
&\tilde{W}_N(A)=\sup_{\sigma\in \R}\Big\{\sigma A -\frac{1}{N} \sum_{i=1}^N \log \int_{\R}\exp(-\tilde{\psi}_i(y)+\sigma y)\,dy\Big\},\label{W_N(A) tilde}
\\&\tilde{F}_N(A)=-\log\int_{\R^{N-1}} \exp\Big[-\sum_{i=1}^N\tilde{\psi}_i(u_i-u_{i-1})\Big]\,du_1\ldots du_{N-1},\label{eq: F_N(A) tilde}
\end{align}
with $u_0=0, u_N=NA$.

Let $\sigma^*$ be the maximizer in \eqref{W_N(A) tilde}. We define the one dimensional probability measures 
\begin{equation}
\label{mu_i}
\tilde{\mu}_i^{\sigma^*}(dy)=Z_i^{-1}\exp(\sigma^* y-\tilde{\psi}_i(y))\,dy,
\end{equation}
where $Z_i$ is the normalising constant. Let $\tilde{X}_i$ be independent random variables distributed according to $\tilde{\mu}_i^{\sigma*}$ and let $\tilde{m}_i$ be the mean of $\tilde{X}_i$. Let $\tilde{g}_{N,A}$ be the density of $\frac{1}{\sqrt{N}}\sum\limits_{i=1}^N (\tilde{X}_i-\tilde{m}_i)$. Then it holds that
\begin{equation}
\label{eq: auxilary eqn}
\tilde{F}_N(A)=\frac{1}{2}\log N+N\tilde{W}_N(A)-\log \tilde{g}_{N,A}(0).
\end{equation}
\end{lemma}
\begin{proof}
This proof is adapted from \cite[Lemma 8]{Menz11} (see also \cite[Eq. (125)]{GOVW09}).
By change of variables $y_i=u_i-u_{i-1}$, for $i=1,\ldots,N-1$, we can re-write $\tilde{F}_N(A)$ as
\begin{equation}
\label{eq: energy 2}
\tilde{F}_N(A)=-\log \int_{\R^{N-1}}\exp\Bigg[-\sum_{i=1}^{N-1} \tilde{\psi}_i(y_i)-\tilde{\psi}_N\bigg(NA-\sum_{i=1}^{N-1}y_i\bigg)\Bigg]\,dy_1\ldots dy_{N-1}.
\end{equation}
We define
\begin{align*}
&\tilde{\varphi}_{N,i}(\sigma)=\log\int_{\R}\exp[-\tilde{\psi}_i(y)+\sigma y]\,dy,
\\&\tilde{\varphi}_N(\sigma):=\frac{1}{N}\log \int_{\R^N}\exp\bigg[-\sum_{i=1}^N\tilde{\psi}_i(y_i)+\sigma\sum_{i=1}^N y_i\bigg]\,dy_1\ldots dy_N
\end{align*} 
then
\begin{equation}
\label{eq: A}
\begin{cases}
\tilde{W}_N(A)=\sigma^* A-\tilde{\varphi}_N(\sigma^*),\\
A=\frac{d}{d\sigma}\tilde{\varphi}_N(\sigma)\Big\vert_{\sigma=\sigma^*}.
\end{cases}
\end{equation}
We have
\begin{align}
\label{eq: varphi_N}
\tilde{\varphi}_N(\sigma)&=\frac{1}{N}\log \int_{\R^N}\exp\bigg[-\sum_{i=1}^N\tilde{\psi}_i(y_i)+\sigma\sum_{i=1}^N y_i\bigg]\,dy_1\ldots dy_N\nonumber
\\&=\frac{1}{N}\log \int_{\R^N}\prod_{i=1}^N\exp\big[-\tilde{\psi}_i(y_i)+\sigma y_i\big]\,dy_1\ldots dy_N\nonumber
\\&=\frac{1}{N}\log \prod_{i=1}^N\int_{\R}\exp\big[-\tilde{\psi}_i(y_i)+\sigma y_i\big]\,dy_i\nonumber
\\&=\frac{1}{N}\sum_{i=1}^N\log\int_{\R}\exp\big[-\tilde{\psi}_i(y_i)+\sigma y_i\big]\,dy_i\nonumber
\\&=\frac{1}{N}\sum_{i=1}^N\tilde{\varphi}_{N,i}(\sigma).
\end{align}
A straightforward calculation gives
\begin{equation}
\label{eq: m_i}
\tilde{m}_i=\int_{\R}y_i\tilde{\mu}_i^{\sigma^*}(dy_i)=\frac{d}{d\sigma}\tilde{\varphi}_{N,i}(\sigma)\Big\vert_{\sigma=\sigma^*}.
\end{equation}
Substituting \eqref{eq: varphi_N} and \eqref{eq: m_i} into \eqref{eq: A}, we obtain
\begin{equation}
\label{eq: A 2}
A=\frac{d}{d\sigma}\tilde{\varphi}_N(\sigma)\Big\vert_{\sigma=\sigma^*}=\frac{1}{N}\sum_{i=1}^N\frac{d}{d\sigma}\tilde{\varphi}_{N,i}(\sigma)\Big\vert_{\sigma=\sigma^*}=\frac{1}{N}\sum_{i=1}^N \tilde{m}_i.
\end{equation}
Since $\tilde{X}_i$ are independent, the density of the sum $\sum_{i=1}^N \tilde{X}_i$ is given by the convolution
\begin{equation*}
\tilde{f}_{\sum_{i=1}^N X_i}(\xi)=(\tilde{\mu}_1^{\sigma^*}\ast\ldots\ast\tilde{\mu}_N^{\sigma^*})(\xi).
\end{equation*}
Using the definition of convolution, we can compute the above density explicitly as follows
\begin{equation*}
\tilde{f}_{\sum_{i=1}^N \tilde{X}_i}(\xi)=\int_{\R^{N-1}}\exp\bigg[-\sum_{i=1}^{N}\tilde{\varphi}_{N,i}(\sigma^*)+\sigma^* \xi -\tilde{\psi}_N(\xi-\sum_{i=1}^{N-1}y_i)-\sum_{i=1}^{N-1}\tilde{\psi}_{i}(y_i)\bigg]dy_1\ldots dy_{N-1}.
\end{equation*}
We recall that if $Y$ has density $f(y)dy$ then, for $\alpha>0, \beta\in\R$, $\alpha Y+\beta$ has density $\frac{1}{\alpha}f(\frac{y-\beta}{\alpha})$. Hence, we obtain
\begin{align*}
&\tilde{g}_{N,A}(\xi)=f_{\frac{1}{\sqrt{N}}\sum_{i=1}^N(\tilde{X}_i-m_i)}(\xi)
\\&=\sqrt{N}\int_{\R^{N-1}}\exp\bigg[-\sum_{i=1}^{N}\tilde{\varphi}_{N,i}(\sigma^*)+\sigma^*\Big(\xi\sqrt{N}+\sum_{i=1}^N\tilde{m}_i\Big) 
\\&\hspace*{4cm}-\tilde{\psi}_N\Big(\sqrt{N}\xi-\sum_{i=1}^{N-1}y_i+\sum_{i=1}^{N}\tilde{m}_i\Big)-\sum_{i=1}^{N-1}\tilde{\psi}_{i}(y_i)\bigg]dy_1\ldots dy_{N-1}.
\end{align*}
In particular, using \eqref{eq: varphi_N}, \eqref{eq: energy 2} and \eqref{eq: A 2}, we get
\begin{align*}
\tilde{g}_{N,A}(0)&=\sqrt{N}\int_{\R^{N-1}}\exp\bigg[-\sum_{i=1}^{N}\tilde{\varphi}_{N,i}(\sigma^*)+\sigma^*\sum_{i=1}^N\tilde{m}_i -\tilde{\psi}_N\Big(-\sum_{i=1}^{N-1}y_i+\sum_{i=1}^{N}\tilde{m}_i\Big)
\\&\hspace*{7cm}-\sum_{i=1}^{N-1}\tilde{\psi}_{i}(y_i)\bigg]dy_1\ldots dy_{N-1}
\\&=\sqrt{N}\int_{\R^{N-1}}\exp\bigg[-N\tilde{\varphi}_{N}(\sigma^*)+\sigma^*N A -\tilde{\psi}_N\Big(NA-\sum_{i=1}^{N-1}y_i\Big)-\sum_{i=1}^{N-1}\tilde{\psi}_{i}(y_i)\bigg]dy_1\ldots dy_{N-1}
\\&=\sqrt{N}\exp[N(\sigma^* A-\tilde{\varphi}_N(\sigma^*))]\int_{\R^{N-1}}\exp\bigg[-\tilde{\psi}_N\Big(NA-\sum_{i=1}^{N-1}y_i\Big)-\sum_{i=1}^{N-1}\tilde{\psi}_{i}(y_i)\bigg]dy_1\ldots dy_{N-1}
\\&=\sqrt{N}\exp[N(\sigma^* A-\tilde{\varphi}_N(\sigma^*))]\exp[-\tilde{F}_N(A)].
\end{align*}
It follows from \eqref{eq: A} and the above equality that
\begin{equation*}
\log \tilde{g}_{N,A}(0)=\frac{1}{2}\log N +N \tilde{W}_N(A)-\tilde{F}_N(A),
\end{equation*}
which is equivalent to \eqref{eq: auxilary eqn} as claimed.
\end{proof}
The following proposition provides an analytical expression of the defect-formation free energy in terms of densities of averages of independent random variables.
\begin{proposition}
\label{prop: difference of energy}
Recall that the Cauchy-Born energy is given by
\begin{equation}
\label{eq: Cauchy-Born}
W(A) = \sup_{\sigma\in \R}\Big\{\sigma A -\log \int_{\R}\exp(-\psi(y) + \sigma y)\,dy\Big\}.
\end{equation}
We define an analogous function that is associated to the defect material
\begin{equation}
\label{W_N^P(A)}
W_N^P(A)=\sup_{\sigma\in \R}\left\{\sigma x - \frac{1}{N} 
\Big( \log \int_{\R}\exp[-(\psi+P)(y)+\sigma y]\,dy 
+ (N-1)\log \int_{\R}\exp(-\psi(y)+\sigma y)\,dy \Big) \right\}.
\end{equation}

Let $\sigma_0$ and $\sigma_P^N$ be the maximisers in definitions of \eqref{eq: Cauchy-Born} and \eqref{W_N^P(A)} respectively. We define the one-dimensional probability measures 
\begin{align}
&\mu^{\sigma_0}(dy)=Z_\mu^{-1}\exp(\sigma_0 y-\psi(y))\,dy,\quad\text{and}\label{mu}
\\&\nu^{\sigma_P^N}(dy)=Z_\nu^{-1}\exp(\sigma_P^N y-(\psi+P)(y))\,dy,\quad  \mu^{\sigma_P^N}(dy)=Z_{\mu_P}^{-1}\exp(\sigma_P^N y-\psi(y))\,dy,\label{nu}
\end{align}
where $Z_\mu, Z_\nu$ and $Z_{\mu_P}$ are normalising constants. Let $m, m_{P,1}$ and $m_{P,2}$ be  respectively the means of $\mu^{\sigma_0}, \nu^{\sigma_P^N}(dy)$ and $\mu^{\sigma_P^N}(dy)$. 

Let $\{X_i\}_{i=1,\ldots, N}$ and $\{Y_i\}_{i=1,\ldots,N}$ be independent random variables, where $\{X_i\}_{i=1,\ldots, N}$ distributed according to $\mu^{\sigma_0}(dy),$ $\{Y_1\}$ distributed according to $\nu^{\sigma_P^N}(dy),$ and $\{Y_i\}_{i=2,\ldots,N}$ distributed according to 
$\mu^{\sigma_P^N}(dy).$  Let $g_{N,A}$ and $g^P_{N,A}$ be respectively the
density of $\frac{1}{\sqrt{N}}\sum_{i=1}^N (X_i-m)$ and
$\frac{1}{\sqrt{N}}\sum_{i=1}^N (Y_i-m_{P,i})$ (with $m_{P,2}=\ldots=m_{P,N}$).

Then it holds that
\begin{equation}
\label{eq: difference of energy}
F_N^P(A)-F_N(A)=N[W_N^P(A)-W(A)]+\log\frac{g_{N,A}(0)}{g^P_{N,A}(0)}.
\end{equation}
\end{proposition}
\begin{proof}Applying Lemma \ref{lem: aulem 1} for the cases $\tilde{\psi}_i=\psi~~ (i=1,\ldots, N)$ and $\tilde{\psi}_1=\psi+P,~~\tilde{\psi}_i=\psi~~ (i=2,\ldots, N)$, we obtain the following relations respectively
\begin{align*}
F_N{A}&=\frac{1}{2}\log N+N W_N(A)-\log g_{N,A}(0),
\\F^P_N(A)&=\frac{1}{2}\log N +N W^P_N(A)-\log g^P_{N,A}(0).
\end{align*}
The assertion \eqref{eq: difference of energy} immediately follows from these two relations.
\end{proof}
The next step is to passing to the limit $N\to\infty$ for each term in the relation \eqref{eq: difference of energy}. We will need some auxiliary lemmas. We define
\begin{align}
&\Psi(\sigma):=\frac{\int_{\R} y\exp(-\psi(y)+\sigma y)\,dy}{\int_{\R} \exp(-\psi(y)+\sigma y)\,dy},\label{eq: Psi}
\\&\Phi(\sigma):=\frac{\int_{\R} y\exp[-(\psi+P)(y)+\sigma y]\,dy}{\int_{\R} \exp[-(\psi+P)(y)+\sigma y]\,dy}-\frac{\int_{\R} y\exp(-\psi(y)+\sigma y)\,dy}{\int_{\R} \exp(-\psi(y)+\sigma y)\,dy}.\nonumber
\end{align}
The following lemma on boundedness of derivatives of $\Psi$ and $\Phi$ will be used several times in the sequel.
\begin{lemma} 
\label{lem: bound of derivatives} 
It holds that
\begin{equation}
\frac{1}{\kappa_2}\leq \frac{d}{d\sigma}\Psi(\sigma)\leq \frac{1}{\kappa_1}\quad \text{and}\quad \Big|\frac{d}{d\sigma}\Phi(\sigma)\Big|\leq C, \label{eq: derivative of Phi}
\end{equation}
for some positive constant $C$.
\end{lemma}
\begin{proof}
We first prove the first part of \eqref{eq: derivative of Phi}. The following proof is simplified from \cite[Lemma 2.4]{Caputo03}. In \cite[Lemma 2.4]{Caputo03} the author has actually proved a stronger result than we need here.  
We have
\begin{align}
\label{eq: derivative of Psi}
\frac{d}{d\sigma}\Psi(\sigma)&=\frac{\left(\int_{\R}y^2\exp(-\psi(y)+\sigma y)\,dy\right)\left(\int_{\R}\exp(-\psi(y)+\sigma y)\,dy\right)-\left(\int_{\R}y\exp(-\psi(y)+\sigma y)\,dy\right)^2}{\left(\int_{\R}\exp(-\psi(y)+\sigma y)\,dy\right)^2}\nonumber
\\&=\int_{\R}\Big(y-m_\sigma\Big)^2\mu_\sigma(dy),
\end{align}
where
\begin{equation*}
\mu_\sigma(dy):=\frac{\exp(-\psi(y)+\sigma y)}{\int_{\R}\exp(-\psi(y)+\sigma y)\,dy}\,dy\in \mathcal{P}(\R),\quad\text{and}~~ m_\sigma=\int_{\R}y\mu_\sigma(dy).
\end{equation*}
Using this equality, we now estimate $\frac{d}{d\sigma}\Psi(\sigma)$ using assumptions on $\psi$. 
For the upper bound: since $\psi''\geq \kappa_1$, $\mu_\sigma$ satisfies the Poincare inequality with constant $\kappa_1$ uniformly in $\sigma$. Therefore,
\[
\frac{d}{d\sigma}\Psi(\sigma)\leq \frac{1}{\kappa_1}\int \Big|\frac{d}{dy} y\Big|^2\mu_\sigma(dy)=\frac{1}{\kappa_1}.
\]
For the lower bound: using the inequality $g^2\geq 2fg-f^2$ for all functions $f$ and $g$, with $g=y-m_\sigma$, we have
\[
\frac{d}{d\sigma}\Psi(\sigma)\geq \int [2f(y-m_\sigma)-f^2]\,\mu_\sigma(dy).
\]
By taking $f=\beta (\psi'-\sigma)$ for $\beta\in\R$, and applying integration by parts, we obtain
\[
\frac{d}{d\sigma}\Psi(\sigma)\geq 2\beta -\beta^2\int \psi''(y)\mu_\sigma(dy).
\]
Now maximizing over $\beta$, by choosing $\beta=\frac{1}{\int \psi''(y)\mu_\sigma(dy)}$, we get
\[
\frac{d}{d\sigma}\Psi(\sigma)\geq \frac{1}{\int \psi''(y)\mu_\sigma(dy)}\geq \frac{1}{\kappa_2},
\]
where we have used the assumption that $\psi''\leq \kappa_2$.

The second estimate in \eqref{eq: derivative of Phi} is proved similarly. We have
\begin{align*}
&\frac{d}{d\sigma}\Phi(\sigma)=\int_{\R}(y-m^P_\sigma)^2\,d\mu^P_\sigma(dx)-\int_{\R}(y-m_\sigma)^2\,d\mu_\sigma(dx),\quad\text{where}
\\&\mu^P_\sigma=\frac{\exp[-(\psi+P)(y)+\sigma y]}{\int_{\R} \exp[-(\psi+P)(y)+\sigma y)\,dy}\,dy,\quad \text{and}\quad m^P_\sigma=\int_{\R}y\mu^P_\sigma(dy).
\end{align*}
Since $\psi+P$ satisfies a similar assumption as $\psi$, we obtain
\[
\frac{1}{\varsigma_2}\leq\int_{\R}(y-m^P_\sigma)^2\,d\mu^P_\sigma(dx)\leq \frac{1}{\varsigma_1}.
\]
As a consequence, we get
\[
\frac{1}{\varsigma_2}-\frac{1}{\kappa_1}\leq\frac{d}{d\sigma}\Phi(\sigma)\leq \frac{1}{\varsigma_1}-\frac{1}{\kappa_2},
\]
which implies the second estimate in \eqref{eq: derivative of Phi}.
\end{proof}

Recalling that $\sigma_0$ and $\sigma_N^P$ are corresponding the maximisers in \eqref{eq: Cauchy-Born} and \eqref{W_N^P(A)}. The following lemma provides an estimate for $|\sigma_N^P-\sigma_0|$.
\begin{lemma}
\label{lem:sigma_P-sigma_0} 
There exists a positive constant $C$ such that, for $N$ sufficiently large,
\begin{equation}
|\sigma_P^N-\sigma_0|\leq \frac{C}{N}.
\end{equation}
\end{lemma}
\begin{proof}
Set $F:=\Psi+\frac{1}{N}\Phi$. Then we have
\begin{equation*}
A=\Psi(\sigma_0)=F(\sigma_P^N),\quad\text{and}\quad F'(\sigma)=\Psi'(\sigma)+\frac{1}{N}\Phi'(\sigma).
\end{equation*}
This, together with Lemma \ref{lem: bound of derivatives},  imply that for sufficiently large $N$ and for all $\sigma\in\R$
\[
\frac{0.5}{\kappa_2}\leq |F'(\sigma)|\leq\frac{2}{\kappa_1}.
\]

By the mean value theorem, there exists $\theta\in \R$ such that
\begin{equation*}
F'(\theta)(\sigma_P^N-\sigma_0)=F(\sigma_P^N)-F_P(\sigma_0)=F_0(\sigma_0)-\left(F_0(\sigma_0)+\frac{1}{N}\Phi(\sigma_0)\right)=-\frac{1}{N}\Phi(\sigma_0).
\end{equation*}
Hence
\begin{equation*}
|\sigma_P^N-\sigma_0^N|=\frac{1}{N}\left|\frac{\Phi(\sigma_0)}{F'(\theta)}\right|\leq \frac{C}{N},
\end{equation*}
for some constant $C>0$ and for $N$ sufficiently large.
\end{proof}
The following estimate is elementary but will be used at various places later.
\begin{lemma}
\label{lem: exponential}
For any $z\in \mathbb{C}$, we have
\begin{equation}
|e^z-1|\leq |z|e^{|z|}.
\end{equation}
\end{lemma}
\begin{proof}
We have
\begin{equation*}
|e^z-1|=\left|e^{tz}\Big\vert_0^1\right|=\left|\int_0^1 ze^{tz}\,dt\right|\leq |z|\int_0^1|e^{tz}|\,dt=|z|\int_0^1 e^{t\mathrm{Rel}(z)}\,dt\leq |z|\int_0^1 e^{|z|}\,dt =|z|e^{|z|}.
\end{equation*}
\end{proof}

The second ingredient of the proof of Theorem \ref{theo: main theorem constrained case} is the following proposition.
\begin{proposition} It holds that
\label{prop: limits of difference of G}
\begin{equation}
\label{limits of difference of G}
\lim_{N\rightarrow \infty} N[W_N^P(A)-W(A)]=-\log\frac{\int \exp[-(\psi+P)(y)-W'(A) y]\,dy}{\int \exp[-\psi(y)-W'(A) y]\,dy}.
\end{equation}
Moreover, it hods that
\begin{equation*}
\Bigg| N[W_N^P(A)-W(A)]+\log\frac{\int \exp[-(\psi+P)(y)-W'(A) y]\,dy}{\int \exp[-\psi(y)-W'(A) y]\,dy}\Bigg|\leq\frac{C}{N}.
\end{equation*}
\end{proposition}
\begin{proof}
We recall that $\sigma_0$ and $\sigma_P^N$ are respectively the maximisers in the definitions of $W(A)$ and $W^P_N(A)$, so that
\begin{align}
W(A)&=\sup_{\sigma\in \R}\Big\{\sigma A -\log \int_{\R}\exp(-\psi(y)+\sigma y)\,dy\Big\}\label{eq: W_N(A) 1}
\\&=\sigma_0 A-\log\int_{\R}\exp(-\psi(y)+\sigma_0 y)\,dy,\label{eq: W_N(A) 2}
\end{align}
where $\sigma_0$ satisfies
\begin{equation}
\label{sigma0}
A=\frac{\int_{\R} y\exp(-\psi(y)+\sigma_0 y)\,dy}{\int_{\R} \exp(-\psi(y)+\sigma_0 y)\,dy}=\Psi(\sigma_0).
\end{equation}
By properties of the Legendre transform, we also have $W'(A)=\sigma_0$, which is explicitly shown in \eqref{eq: Legendre transform}. Similarly
\begin{align}
W_N^P(A)&=\sup_{\sigma\in \R}\Bigg\{\sigma A - \frac{1}{N} 
\bigg( \log \int_{\R}\exp[-(\psi+P)(y)+\sigma y]\,dy 
+  (N-1) \log \int_{\R}\exp(-\psi(y)+\sigma y)\,dy \bigg) \Bigg\}\label{eq: W^P_N(A) 1}
\\&=\sigma_P^N A-\log\int_{\R}\exp[-\psi(y)+\sigma_P^N y]\,dy-\frac{1}{N}\log\frac{\int_{\R}\exp[-(\psi+P)(y)+\sigma_P^N y]\,dy}{\int_{\R}\exp(-\psi(y)+\sigma_P^N y)\,dy}\label{eq: W^P_N(A) 2},
\end{align}
where $\sigma_P^N$ solves
\begin{equation}
\label{sigmaP}
A=\frac{1}{N}\frac{\int_{\R}y\exp[-(\psi+P)(y)+\sigma y]\,dy}{\int_{\R}\exp[-(\psi+P)(y)+\sigma y]\,dy}+\frac{N-1}{N}\frac{\int_{\R} y\exp(-\psi(y)+\sigma y)\,dy}{\int_{\R} \exp(-\psi(y)+\sigma y)\,dy}.
\end{equation}
Using these supremum representations we will estimate lower and upper bounds for $N[W_N^P(A)-W(A)]$. 
For an upper bound: it follows from \eqref{eq: W_N(A) 1} that
\begin{equation*}
W_N(A)\geq\sigma_P^N A-\log\int_{\R}\exp(-\psi(y)+\sigma_P^N y)\,dy.
\end{equation*} 
This, together with \eqref{eq: W^P_N(A) 2}, we get
\begin{equation*}
N[W_N^P(A)-W(A)]\leq -\log\frac{\int_{\R}\exp[-(\psi+P)(y)+\sigma_P^N y]\,dy}{\int_{\R}\exp(-\psi(y)+\sigma_P^N y)\,dy}.
\end{equation*}
Similarly, using \eqref{eq: W^P_N(A) 1} and \eqref{eq: W_N(A) 2}, we obtain 
\begin{equation*}
N[W_N^P(A)-W(A)]\geq -\log\frac{\int_{\R}\exp[-(\psi+P)(y)+\sigma_0 y]\,dy}{\int_{\R}\exp(-\psi(y)+\sigma_0 y)\,dy}.
\end{equation*}
Bringing these bounds together,
\begin{equation}
\label{eq: lower and upper bound}
-\log\frac{\int_{\R}\exp[-(\psi+P)(y)+\sigma_0 y]\,dy}{\int_{\R}\exp(-\psi_1(y)+\sigma_0 y)\,dy}\leq N[W_N^P(A)-W(A)]\leq -\log\frac{\int_{\R}\exp[-(\psi+P)(y)+\sigma_P^N y]\,dy}{\int_{\R}\exp(-\psi(y)+\sigma_P^N y)\,dy}.
\end{equation}
We now estimate the right-hand side of the last expression. We have
\begin{align*}
&\frac{\int_{\R}\exp[-(\psi+P)(y)+\sigma_P^N y]\,dy}{\int_{\R}\exp(-\psi(y)+\sigma_P^N y)\,dy}
\\&\qquad=\frac{\int_{\R}\exp[-(\psi+P)(y)+\sigma_0 y+(\sigma_P^N-\sigma_0) y]\,dy}{\int_{\R}\exp[-(\psi+P)(y)+\sigma_0 y)\,dy}
\times \frac{\int_{\R}\exp[-(\psi+P)(y)+\sigma_0 y]\,dy}{\int_{\R}\exp(-\psi(y)+\sigma_0 y)\,dy}
\\&\qquad\qquad\times \frac{\int_{\R}\exp[-\psi(y)+\sigma_0 y]\,dy}{\int_{\R}\exp[-\psi(y)+\sigma_P^N y]\,dy}
\\&\qquad=\frac{\int_{\R}\exp[-(\psi+P)(y)+\sigma_0 y]\,dy}{\int_{\R}\exp(-\psi_1(y)+\sigma_0 y)\,dy}\times\left\langle\exp[(\sigma_P^N-\sigma_0) y]\right\rangle_{\nu^{\sigma_0}}\times\left\langle\exp[(\sigma_P^N-\sigma_0) y]\right\rangle_{\mu^{\sigma_0}}^{-1},
\end{align*}
where
\begin{align*}
\nu^{\sigma_0}(y)dy=\frac{\exp[-(\psi+P)(y)+\sigma_0 y]}{\int_{\R}\exp[-(\psi+P)(y)+\sigma_0 y)]\,dy}\,dy
\quad \text{and}\quad\mu^{\sigma_0}(y)dy=\frac{\exp[-\psi(y)+\sigma_0 y]}{\int_{\R}\exp[-\psi(y)+\sigma_0 y]\,dy}.
\end{align*}
Taking the logarithm of the above equality, we deduce
\begin{align}
\label{eq: temp estimate}
&-\log \frac{\int_{\R}\exp[-(\psi+P)(y)+\sigma_P^N y]\,dy}{\int_{\R}\exp(-\psi(y)+\sigma_P^N y)\,dy}\nonumber
\\&\quad=-\log\frac{\int_{\R}\exp[-(\psi+P)(y)+\sigma_0 y]\,dy}{\int_{\R}\exp(-\psi(y)+\sigma_0 y)\,dy}+\log\left\langle\exp[(\sigma_P^N-\sigma_0) y]\right\rangle_{\nu^{\sigma_0}}-\log\left\langle\exp[(\sigma_P^N-\sigma_0) y]\right\rangle_{\mu^{\sigma_0}}.
\end{align}
We now show that the last two terms in the right-hand side of \eqref{eq: temp estimate} are of order $O(N^{-1})$. Using the estimate $|e^t-1|\leq |t|e^{|t|}$ (Lemma \ref{lem: exponential}) and Lemma \ref{lem:sigma_P-sigma_0}, we have
\begin{equation*}
|\exp[(\sigma_P^N-\sigma_0) y]-1|\leq |(\sigma_P^N-\sigma_0) y|\exp(|(\sigma_P^N-\sigma_0) y|)\leq \frac{C}{N}|y|\exp(C|y|).
\end{equation*}
Therefore
\begin{align*}
\left|\left\langle\exp[(\sigma_P^N-\sigma_0) y]\right\rangle_{\nu^{\sigma_0}}-1\right|&=\left|\left\langle\exp[(\sigma_P^N-\sigma_0) y]-1\right\rangle_{\nu^{\sigma_0}}\right|\leq\left\langle|
\exp[(\sigma_P^N-\sigma_0) y]-1|\right\rangle_{\nu^{\sigma_0}}
\\&\leq \frac{C}{N}\langle|y|\exp(C|y|)\rangle_{\nu^{\sigma_0}}.
\end{align*}
Since $(\psi+P)(y)$ is bounded from below and above by a quadratic potential, it implies that the term
\begin{align*}
\langle|y|\exp(C|y|)\rangle_{\nu^{\sigma_0}}=\frac{1}{\int_{\R}\exp[-(\psi+P)(y)+\sigma_0 y]\,dy}\int |y|\exp[-(\psi+P)(y)+\sigma_0 y+C|y|]\,dy.
\end{align*}
is finite. Therefore $\left|\left\langle\exp[(\sigma_P^N-\sigma_0) y]\right\rangle_{\nu^{\sigma_0}}-1\right|\leq \frac{C}{N}$, which implies that
\begin{equation*}
\left|\log \left\langle\exp[(\sigma_P^N-\sigma_0) y]\right\rangle_{\nu^{\sigma_0}}\right|\leq \frac{C}{N}.
\end{equation*}
Similarly, we obtain the following estimate for the last term in \eqref{eq: temp estimate}
\begin{equation*}
\left|\log \left\langle\exp[(\sigma_P^N-\sigma_0) y]\right\rangle_{\mu^{\sigma_0}}\right|\leq \frac{C}{N}.
\end{equation*}
Substituting these above estimates to \eqref{eq: temp estimate},  we achieve the following estimate for the upper bound in \eqref{eq: lower and upper bound}
\begin{equation*}
\Bigg|-\log \frac{\int_{\R}\exp[-(\psi+P)(y)+\sigma_P^N y]\,dy}{\int_{\R}\exp(-\psi(y)+\sigma_P^N y)\,dy}+\log\frac{\int_{\R}\exp[-(\psi+P)(y)+\sigma_0 y]\,dy}{\int_{\R}\exp(-\psi(y)+\sigma_0 y)\,dy}\Bigg|\leq\frac{C}{N}.
\end{equation*}
Therefore, it follows from \eqref{eq: lower and upper bound} that
\begin{equation*}
\Bigg|N[W_N^P(A)-W(A)]+\log\frac{\int \exp[-(\psi+P)(y)-\sigma_0 y]\,dy}{\int \exp[-\psi(y)-\sigma_0 y]\,dy}\Bigg|\leq\frac{C}{N}.
\end{equation*}
This completes the proof of the proposition.
\end{proof}
Next, we estimate the last term in \eqref{eq: difference of energy}. We will need two auxiliary lemmas.

Let $h(m,\xi):=\left\langle \exp(i\xi(x-m))\right\rangle_{\mu^\sigma}$, where $\mu^\sigma(x)\,dx =Z_\sigma^{-1}\exp(\sigma x-\psi(x))\,dx$. 
\begin{lemma} 
\label{lem: estimate of lambda}
For any $\delta>0$, it holds that
\begin{equation}
|h(m,\xi)|\leq 1-\frac{1}{2}\sqrt{C_\sigma}\, \Bigg(1-\exp\bigg(-\frac{\delta^2}{2\kappa_2}\bigg)\Bigg)\quad\text{for}~~|\xi|\geq\delta,
\end{equation}
where $C_\sigma=\exp\left(\frac{\sigma^2}{4}\frac{\kappa_1-\kappa_2}{\kappa_1\kappa_2}\right)\sqrt{\frac{\kappa_1}{\kappa_2}}$.
\end{lemma}
Note that since $0<\kappa_1<\kappa_2$, we have $0<C_\sigma<\sqrt{\frac{\kappa_1}{\kappa_2}}<1$, which is independent of $\sigma$.
\begin{proof}
The proof of this lemma is adapted from that of \cite[Lemma 39, (i)]{GOVW09}. Since $\kappa_1 x^2\leq\psi(x)\leq \kappa_2 x^2$, we have
\begin{equation*}
\mu^\sigma(x)\geq \frac{\exp(\sigma x-\kappa_2 x^2)}{\int_{\R}\exp(\sigma y-\kappa_1 y^2)\,dy}=\frac{\exp(\sigma x-\kappa_2 x^2)}{\int_{\R}\exp(\sigma y-\kappa_2 y^2)\,dy}\frac{\int_{\R}\exp(\sigma y-\kappa_2 y^2)\,dy}{\int_{\R}\exp(\sigma y-\kappa_1 y^2)\,dy}=n_\sigma(x)C_\sigma,
\end{equation*}
where
\begin{equation*}
n_\sigma(x)=\frac{\exp(\sigma x-\kappa_2 x^2)}{\int_{\R}\exp(\sigma y-\kappa_2 y^2)\,dy},\quad C_\sigma=\frac{\int_{\R}\exp(\sigma y-\kappa_2 y^2)\,dy}{\int_{\R}\exp(\sigma y-\kappa_1 y^2)\,dy}=\exp\left(\frac{\sigma^2}{4}\frac{\kappa_1-\kappa_2}{\kappa_1\kappa_2}\right)\sqrt{\frac{\kappa_1}{\kappa_2}}.
\end{equation*}
Note that $0<C_\sigma<1$ for all $\sigma$. The following identity is the same as \cite[(157)]{GOVW09}
\begin{equation}
|h(m,\xi)|^2=1-\mathrm{Var}(\cos(\xi x))-\mathrm{Var}(\sin(\xi x)).
\end{equation} 
Next we estimate $\mathrm{Var}(\cos(\xi x))$.
\begin{align}
\label{eq: varcos}
\mathrm{Var}(\cos(\xi x))&=\int_{\R}\left(\cos(\xi x)-\int_{\R}\cos(\xi y)\mu_\sigma\,dy\right)^2\mu_\sigma\,dy\nonumber
\\&\geq C_\sigma\int_{\R}\left(\cos(\xi x)-\int_{\R}\cos(\xi y)\mu_\sigma\,dy\right)^2n_\sigma(x)\nonumber
\\&\geq C_\sigma\int_{\R}\left[\int_{\R}\cos(\xi x)^2n_\sigma(dx)-\left(\int_{\R}\cos(\xi x)n_\sigma(dx)\right)^2\right]
\end{align}
The second integral on the right-hand side can be computed explicitly as follows:
\begin{align*}
&\left(\int_{\R}\cos(\xi y)n_\sigma(dy)\right)^2
\\&\quad=\frac{1}{4}\left(\sqrt{\frac{\kappa_2}{\pi}}\exp\bigg(-\frac{\sigma^2}{4\kappa_2}\bigg)\int_{\R}[\exp(i\xi x)+\exp(-i\xi x)]\exp(-\kappa_2 x^2+\sigma x)\,dx\right)^2
\\&\quad=\frac{1}{4}\left(\sqrt{\frac{\kappa_2}{\pi}}\exp\bigg(\frac{i \sigma \xi}{2\kappa_2}\bigg)\int_{\R}\exp(i\xi y)\exp(-\kappa_2 y^2)\,dy+\sqrt{\frac{\kappa_2}{\pi}}\exp\bigg(-\frac{i \sigma \xi}{2\kappa_2}\bigg)\int_{\R}\exp(-i\xi y)\exp(-\kappa_2 y^2)\,dy\right)^2
\\&\quad=\frac{1}{4}\left(\exp\bigg(-\frac{\xi^2}{4\kappa_2}\bigg)\exp\bigg(\frac{i\sigma \xi}{2\kappa_2}\bigg)+\exp\bigg(-\frac{\xi^2}{4\kappa_2}\bigg)\exp\bigg(-\frac{i\sigma \xi}{2\kappa_2}\bigg)\right)^2
\\&\quad=\frac{1}{4}\exp\bigg(-\frac{\xi^2}{2\kappa_2}\bigg)\left(\exp\bigg(\frac{i\sigma \xi}{\kappa_2}\bigg)+\exp\bigg(-\frac{i\sigma \xi}{\kappa_2}\bigg)+2\right)
\\&\quad=\frac{1}{2}\exp\bigg(-\frac{\xi^2}{2\kappa_2}\bigg)\left(1+\cos(\frac{\sigma \xi}{\kappa_2})\right).
\end{align*}
The first integral can be computed similarly:
\begin{equation*}
\int_{\R}\cos^2(\xi x)n_\sigma(dx)=\frac{1}{2}\left(1+\cos(\frac{\sigma \xi}{\kappa_2})\exp(-\frac{\xi^2}{\kappa_2})\right).
\end{equation*}
Therefore,
\begin{align}
&\int_{\R}\cos(\xi x)^2n_\sigma(dx)-\left(\int_{\R}\cos(\xi x)n_\sigma(dx)\right)^2\nonumber
\\&\qquad=\frac{1}{2}\Bigg(1-\exp\bigg(-\frac{\xi^2}{2\kappa_2}\bigg)\Bigg)\Bigg(1-\cos\bigg(\frac{\sigma \xi}{\kappa_2}\bigg)\exp\bigg(-\frac{\xi^2}{2\kappa_2}\bigg)\Bigg)\nonumber
\\&\qquad \geq \frac{1}{2}\Bigg(1-\exp\bigg(\frac{-\xi^2}{2\kappa_2}\bigg)\Bigg)^2.\nonumber
\end{align}
Substituting these computations into \eqref{eq: varcos} we obtain
\begin{align*}
\mathrm{Var}(\cos(\xi x))\geq \frac{1}{2}C_\sigma \Bigg(1-\exp\bigg(-\frac{\xi^2}{2\kappa_2}\bigg)\Bigg)^2.
\end{align*}
By repeating the computation, we obtain that the same inequality holds for $\mathrm{Var}(\sin(\xi x))$. Therefore,
\begin{equation*}
|h(m,\xi)|^2\leq 1-C_\sigma\, \Bigg(1-\exp\bigg(-\frac{\xi^2}{2\kappa_2}\bigg)\Bigg)^2.
\end{equation*}
If $|\xi|\geq \delta$, then
\begin{equation*}
|h(m,\xi)|^2\leq 1-C_\sigma\, \Bigg(1-\exp\bigg(-\frac{\delta^2}{2\kappa_2}\bigg)\Bigg)^2.
\end{equation*}
Since $\sqrt{1-x}\leq 1-\frac{1}{2}x$, it follows that
\begin{equation*}
|h(m,\xi)|\leq 1-\frac{1}{2}\sqrt{C_\sigma}\, \Bigg(1-\exp\bigg(-\frac{\delta^2}{2\kappa_2}\bigg)\Bigg) \qquad\text{for}~~|\xi|\geq \delta.
\end{equation*}
This concludes the proof.
\end{proof}
Define $\Lambda(\sigma):=\mathrm{Var}(\mu_\sigma)=\int_{\R}\left(x-\int_{\R}x\,\mu_\sigma(dx)\right)^2\mu_\sigma(dx)$.
\begin{lemma} 
\label{lem: difference of variance}
There exists $C>0$ such that, for any $\sigma_1,\sigma_2\in \R$, 
\begin{equation*}
|\Lambda(\sigma_1)-\Lambda(\sigma_2)|\leq C|\sigma_1-\sigma_2|.
\end{equation*}
\end{lemma}
\begin{proof}
It follows from \eqref{eq: derivative of Psi} that $\Lambda(\sigma)=\Psi'(\sigma)$. According to \cite[Lemma 41]{GOVW09} we have
\begin{equation*}
|\Psi''(\sigma)|\leq C,
\end{equation*}
for some constant $C>0$. As a consequence, we obtain that
\begin{equation*}
|\Lambda(\sigma_1)-\Lambda(\sigma_2)|=|\Psi'(\sigma_1)-\Psi'(\sigma_2)|\leq C|\sigma_1-\sigma_2|.
\end{equation*}
This finishes the proof.
\end{proof}
We are now ready to estimate the last term in the right-hand side of \eqref{eq: difference of energy}.
\begin{proposition}
\label{prop: limit of ratio of densities}
There exists $C>0$ such that
\begin{equation}
\label{limit of ratio of densities}
\Bigg|\log\frac{g^P_{N,A}(0)}{g_{N,A}(0)}\Bigg|\leq\frac{C}{N}.
\end{equation}
\end{proposition}
\begin{proof}
We recall the general setting in Lemma \ref{lem: aulem 1}.
\begin{equation*}
\tilde{\mu}_j^{\sigma^*}(dy)=\exp\big[-\tilde{\varphi}_{N,j}(\sigma^*)+\sigma^* y-\tilde{\psi}_j(y)\big]\,dy,
\end{equation*}
where
\begin{equation*}
\tilde{\varphi}_{j}(\sigma)=\log\int_{\R}\exp[-\tilde{\psi}_j(y)+\sigma\, y]\,dy
\end{equation*}
For each $j=1,\ldots, N$, let $\tilde{m}_j$ and $\tilde{\varsigma}_j^2$ be the mean and variance of $\tilde{\mu}_j^{\sigma^*}$, i.e., 
\begin{equation*}
\tilde{m}_j=\int_{\R}y\tilde{\mu}_j^{\sigma^*}(dy) \quad\text{and}\quad \tilde{\varsigma}_j^2=\int_{\R}(y-\tilde{m}_j)^2\,\tilde{\mu}_j^{\sigma^*}(dy).
\end{equation*}
Then $\tilde{g}_{N,A}$ has been defined to be the density of $\frac{1}{\sqrt{N}}\sum_{j=1}^N (\tilde{X}_j-m_j)$, where $\tilde{X}_j$ are independent random variables distributed according to $\tilde{\mu}_j^{\sigma^*}$. 

Define $\tilde{y}_j=y_j-\tilde{m}_j$. The value of $\tilde{g}_{N,A}$ at $0$ can be expressed as (cf. e.g.,\cite[Eq. (127)]{GOVW09},\cite[Eq. (72)]{Menz11})
\begin{equation*}
\tilde{g}_{N,A}(0)=\frac{1}{2\pi}\int_{\R} \prod_{j=1}^N\left\langle\exp\Big(i\, \tilde{y}_j\,\frac{1}{\sqrt{N}}\,\xi\Big)\right\rangle_j\,d\xi,
\end{equation*}
where $\langle\cdot\rangle_j$ denotes the average with respect to $\tilde{\mu}_j^{\sigma^*}$. For some $\delta>0$ sufficiently small, we split the above integral into two terms
\begin{align*}
\int_{\R} \prod_{j=1}^N\left\langle\exp\Big(i\, \tilde{y}_j\,\frac{1}{\sqrt{N}}\,\xi\Big)\right\rangle_j\,d\xi&=\int_{\Big\{\big|\frac{1}{\sqrt{N}}\xi\big|\leq \delta\Big\}}\prod_{j=1}^N\Big\langle\exp\Big(i\, \tilde{y}_j\,\frac{1}{\sqrt{N}}\,\xi\Big)\Big\rangle_j\,d\xi
\\&\qquad+\int_{\Big\{\big|\frac{1}{\sqrt{N}}\xi\big|\geq \delta\Big\}}\prod_{j=1}^N\Big\langle\exp\Big(i\, \tilde{y}_j\,\frac{1}{\sqrt{N}}\,\xi\Big)\Big\rangle_j\,d\xi
\\&=\rm{I}+\rm{II},
\end{align*}
so that 
\begin{equation}
\label{f_Nx as Fourier trans }
\tilde{g}_{N,A}(0)=\frac{1}{2\pi}(\rm{I}+\rm{II}).
\end{equation}
According to \cite[Proof of Theorem 4]{Menz11}, the following estimates hold
\begin{equation}
\label{eq: estimate Is}
0<C_1\leq |\rm{I}|\leq C_2,\quad \text{and}\quad |\rm{II}|\leq C_3\,N\,\lambda^{N-2},
\end{equation}
for some positive constants $C_1,C_2, C_3$ and $0<\lambda<1$ depending only on $\delta$. The constant $\lambda$ is the upper bound of $\big|\langle\exp(i\tilde{y}_j\xi)\rangle_{j}\big|$. Moreover,  there exists a complex-valued function $h_j(\xi)$ such that for $0<|\xi|$ sufficiently small,
\begin{equation}
\label{eq: hj}
\langle\exp(i \tilde{y}_j \xi)\rangle_j=\exp(-h_j(\xi)) \quad\text{with}\quad \Big| h_j(\xi)-\frac{1}{2}\tilde{\varsigma}_j^2\xi^2\Big|\leq C |\xi|^3.
\end{equation}
We are now ready to prove Proposition \ref{prop: limit of ratio of densities}.
Applying \eqref{f_Nx as Fourier trans }, \eqref{eq: estimate Is} and \eqref{eq: hj} for the perfect material, we have
\begin{equation*}
g_{N,A}(0)=\frac{1}{2\pi}(\rm{I}_1+\rm{II}_1),
\end{equation*}
where
\begin{align}
&\rm{I}_1=\int_{\Big\{\big|\frac{1}{\sqrt{N}}\xi\big|\leq\delta\Big\}}\exp\Big(-N h(\frac{\xi}{\sqrt{N}})\Big)\,d\xi,\label{I1}
\\& 0<C_{11}\leq |\rm{I}_1|\leq C_{12},\quad\text{and}\quad |\rm{II}_1|\leq C_{13} N\lambda_1^{N-2},\label{estimate1}
\end{align}
for some $0<\lambda_1<1$ and positive constants $C_{11}, C_{12}, C_{13}$ and $\Big| h(\xi)-\frac{1}{2}\varsigma^2\xi^2\Big|\leq C |\xi|^3$ for $|\xi|\ll 1$ with $\varsigma^2$ denotes the variance of $\mu^{\sigma_0}$. According to Lemma \ref{lem: estimate of lambda}, the constant $\lambda_1$ is given by
\begin{equation*}
\lambda_1=1-\frac{1}{2}\sqrt{C_{\sigma_0}}\Bigg(1-\exp\bigg(-\frac{\delta^2}{2\kappa_2}\bigg)\Bigg),
\end{equation*}
with $0<C_{\sigma_0}<1$.
Similarly, 
\begin{equation}
g^P_{N,A}(0)=\frac{1}{2\pi}(\rm{I}_2+\rm{II}_2),
\end{equation}
where
\begin{align}
&\rm{I}_2=\int_{\Big\{\big|\frac{1}{\sqrt{N}}\xi\big|\leq\delta\Big\}}\exp\Bigg(-\sum_{j=1}^N \tilde{h}_j\bigg(\frac{\xi}{\sqrt{N}}\bigg)\Bigg)\,d\xi,\label{I2}
\\& 0<C_{21}\leq |\rm{I}_2|\leq C_{22},\quad\text{and}\quad |\rm{II}_2|\leq C_{23} N\lambda_2^{N-2},\label{estimate2}
\end{align}
for some $0<\lambda_2<1$ and positive constants $C_{21}, C_{22}, C_{23}$ and
\begin{align*}
&\Big|\tilde{h}_1(\xi)-\frac{1}{2}\varsigma_{P,1}^2\xi^2\Big|\leq C |\xi|^3, \quad\text{for}~~|\xi|~~\text{sufficiently small}, 
\\&\tilde{h_2}=\ldots=\tilde{h}_N, \quad\varsigma_{P,2}=\ldots=\varsigma_{P,N}, 
\Big|\tilde{h}_j(\xi)-\frac{1}{2}\varsigma_{P,j}^2\xi^2\Big|\leq C |\xi|^3, \quad\text{for}~~|\xi|~~\text{sufficiently small}, 
\end{align*}
where $\zeta_{P,1}^2$ and $\zeta_{P,2}^2$ are respectively the variances of $\nu^{\sigma_P^N}$ and $\mu^{\sigma_P^N}$.

The constant $\lambda_2$ is given by
\begin{equation*}
\lambda_2=\max\Bigg\{1-\frac{1}{2}\sqrt{C_{\sigma_P^N}}\bigg(1-\exp\Big(-\frac{\delta^2}{\kappa_2}\Big)\bigg),1-\frac{1}{2}\sqrt{\tilde{C}_{\sigma_P^N}}\bigg(1-\exp\Big(-\frac{\delta^2}{\kappa_2+\varsigma_2}\Big)\bigg)\Bigg\},
\end{equation*}
with $0<C_{\sigma_P^N},\tilde{C}_{\sigma_P^N}<1$.

Hence we obtain
\begin{align}
\label{eq: ratio}
\frac{g^P_{N,A}(0)}{g_{N,A}(0)}-1=\frac{\rm{I}_2+\rm{II}_2}{\rm{I}_1+\rm{II}_1}-1=\frac{\rm{I}_2-\rm{I}_1}{\rm{I}_1+\rm{II}_1}+\frac{\rm{II}_2-\rm{II}_1}{\rm{I}_1+\rm{II}_1}.
\end{align}
It follows from \eqref{estimate1} that $|I_1+II_1|\leq C$ for $N$ sufficiently large, thus
\begin{equation}
\label{eq: ratio g-1}
\left|\frac{g^P_{N,A}(0)}{g_{N,A}(0)}-1\right|\leq |\rm{I}_2-\rm{I}_1|+|\rm{II}_2-\rm{II}_1|.
\end{equation}
The second term decays exponentially fast since, from \eqref{estimate1} and \eqref{estimate2}
\begin{align}
|\rm{II}_2-\rm{II}_1|&\leq |\rm{II}_1|+|\rm{II}_2|\leq C N \lambda^{N-2},\label{eq: the first term}
\end{align}
with $\lambda=\max\{\lambda_1,\lambda_2\}$. It follows that $\lambda=1-O(\delta^2)$. 

It remains to estimate $|\rm{I}_2-\rm{I}_1|$. By changing variable $t:=\frac{\xi}{\sqrt{N}}$, we have

\begin{align}
\label{eq: I1-I2}
\rm{I}_1-\rm{I}_2&=\int_{\Big\{\big|\frac{1}{\sqrt{N}}\xi\big|\leq\delta\Big\}}\Bigg[\exp\bigg(-Nh\Big(\frac{\xi}{\sqrt{N}}\Big)\bigg)-\exp\bigg(-\sum_{j=1}^N \tilde{h}_j\Big(\frac{\xi}{\sqrt{N}}\Big)\bigg)\Bigg]\,d\xi\nonumber
\\&=\sqrt{N}\int_{-\delta}^\delta \bigg[\exp(-Nh(t))-\exp\Big(-\sum_{j=1}^N \tilde{h}_j(t)\Big)\bigg]\,dt\nonumber
\\&=\sqrt{N}\int_{-\delta}^\delta \exp\Big(-N h(t)\Big)\bigg(1-\exp\Big(\sum_{j=1}^N (h(t)-\tilde{h}_j(t))\Big)\bigg)\,dt.
\end{align}
Note that
\begin{align*}
&\Big|h(t)-\frac{1}{2N}\zeta^2 t^2\Big|\leq C \frac{t^3}{N^\frac{3}{2}},\quad \Big|\tilde{h}_1(t)-\frac{1}{2N}\zeta_{P,1}^2 t^2\Big|\leq C \frac{t^3}{N^\frac{3}{2}},
\\& \tilde{h}_j(t)=\ldots =\tilde{h}_N(t),\quad \zeta_{P,j}=\zeta_{P,2} \quad\text{for}~~ j=2,\ldots,N,\quad\text{and}
\\&\Big|\tilde{h}_j(t)-\frac{1}{2N}\zeta_{P,2}^2 t^2\Big|\leq C \frac{t^3}{N^\frac{3}{2}},
\end{align*}
where we recall that
$\zeta^2, \zeta_{P,1}^2$ and $\zeta_{P,2}^2$ are, respectively, the variances of $\mu^{\sigma_0}$, $\nu^{\sigma_P^N}$ and $\mu^{\sigma_P^N}$.
It follows that, for $t<1$,
\begin{align}
\label{eq: temp est 1}
\Big|\exp\big(-N h(t)\big)\Big|&=\exp\Big(-\frac{1}{2}\zeta^2 t^2 \Big)\Bigg|\exp\bigg(-N \Big(h(t)-\frac{1}{2N}\zeta^2t^2\Big)\bigg)\Bigg|\nonumber
\\&\leq \exp\bigg(-\frac{1}{2}\zeta^2 t^2 \bigg)\exp\bigg(\frac{Ct^3}{N^{\frac{1}{2}}}\bigg)\nonumber
\\&\leq \exp\bigg(\frac{Ct^2}{N^{\frac{1}{2}}}\bigg).
\end{align}
Now we estimate 
\begin{align}
\label{eq: difference of sum of h}
&\bigg|\sum_{j=1}^N (h(t)-\tilde{h}_j(t))\bigg|\nonumber
\\&\qquad=\bigg|\sum_{j=1}^N\left(h(t)-\frac{1}{2N}\zeta^2 t^2+\frac{1}{2N}\zeta^2 t^2-\frac{1}{2N}\zeta_{P,j}^2 t^2+\frac{1}{2N}\zeta_{P,j}^2 t^2-\tilde{h}_j(t)\right)\bigg|\nonumber
\\&\qquad\leq\sum_{j=1}^N\left[\Big|h(t)-\frac{1}{2N}\zeta^2 t^2\Big|+\Big|\frac{1}{2N}\zeta^2 t^2-\frac{1}{2N}\zeta_{P,j}^2 t^2\Big|+\Big|\frac{1}{2N}\zeta_{P,j}^2 t^2-\tilde{h}_j(t)\Big|\right] \nonumber
\\&\qquad\leq \frac{Ct^3}{N^\frac{1}{2}}+\frac{N-1}{2N}\big|\zeta^2-\zeta_{P,2}^2\big|t^2+\frac{1}{2N}\big|\zeta^2-\zeta_{P,1}^2\big|t^2.
\end{align}
From Lemma \ref{lem:sigma_P-sigma_0} and Lemma \ref{lem: difference of variance}, we have
\begin{align*}
&\big|\zeta^2-\zeta_{P,2}^2\big|=\big|\Lambda(\sigma_0)-\Lambda(\sigma_P^N)\big|\leq C|\sigma_0-\sigma_P^N|\leq \frac{C}{N},\quad\text{and}
\\&|\zeta^2-\zeta_{P,1}^2|=|\Lambda_P(\sigma_P^N)-\Lambda(\sigma_0)|\leq |\Lambda_P(\sigma_P^N)-\Lambda_P(\sigma_0)|+|\Lambda_P(\sigma_0)-\Lambda(\sigma_0)|\leq \frac{C}{N}+C,
\end{align*}
where $\Lambda_P(\sigma)$ is the variance of the measure $Z^{-1}\int \exp[-(\psi+P)(x)+\sigma x]\,dx$ and the last inequality is obtained similarly as in Lemma \ref{lem: difference of variance}.

Substituting these estimates into \eqref{eq: difference of sum of h}, we obtain that, for $t<1$,
\begin{equation*}
\Big|\sum_{j=1}^N (h(t)-\tilde{h}_j)(t)\Big|\leq \frac{Ct^3}{N^\frac{1}{2}}+\frac{C t^2}{N}+\frac{Ct^2}{N^2}\lesssim \frac{C t^2}{N^\frac{1}{2}}.
\end{equation*}
Therefore by using the estimate $|e^z-1|\leq |z|e^{|z|}$, we obtain
\begin{align}
\label{eq: temp est 2}
\Bigg|1-\exp\bigg(\sum_{j=1}^N (h(t)-\tilde{h}_j(t))\bigg)\Bigg|&\leq \bigg|\sum_{j=1}^N (h(t)-\tilde{h}_j(t))\bigg|\exp\bigg(\Big|\sum_{j=1}^N (h(t)-\tilde{h}_j(t))\Big|\bigg)\nonumber
\\&\leq\frac{Ct^2}{\sqrt{N}}\exp\bigg(\frac{Ct^2}{\sqrt{N}}\bigg).
\end{align}
Substituting the estimates \eqref{eq: temp est 1}-\eqref{eq: temp est 2} into \eqref{eq: I1-I2}, we obtain
\begin{align*}
|\rm{I}_1-\rm{I}_2|&\leq \sqrt{N}\int_{-\delta}^{\delta}\exp\Big(\frac{C t^2}{N^\frac{1}{2}}\Big)\frac{Ct^2}{\sqrt{N}}\exp\Big(\frac{Ct^2}{\sqrt{N}}\Big)\,dt
\\&\leq C\exp\Big(\frac{C \delta^2}{N^\frac{1}{2}}\Big)\int_{-\delta}^\delta t^2\,dt=O(\delta^3).
\end{align*}
By choosing $\delta=N^{-\alpha}$ where $\frac{1}{3}<\alpha<\frac{1}{2}$ then
\begin{align*}
&|II_2-II_1|\lesssim N\lambda^N\lesssim N(1-N^{-2\alpha})^N\lesssim N\Big(e^{-N^{-2\alpha}}\Big)^N=N e^{-N^{-2\alpha+1}}\lesssim N^{-1},
\\&|I_1-I_2|\lesssim N^{-3\alpha}\lesssim N^{-1}.
\end{align*}
Substituting these estimates into \eqref{eq: ratio g-1}, we obtain
\begin{equation*}
\left|\frac{g^P_{N,A}(0)}{g_{N,A}(0)}-1\right|\lesssim N^{-1},
\end{equation*}
implying that
\begin{equation*}
\left|\log\frac{g^P_{N,A}(0)}{g_{N,A}(0)}\right|\lesssim N^{-1}.
\end{equation*}
This completes the proof of the proposition.
\end{proof}
\subsection{Coarse-grained energy}
In this section, we prove Theorem \ref{th:intro:cg} for the case without external forces by deriving the formula for the coarse-grained energy and the representation of the thermodynamic limit $G_\infty(A)$. 

We recall that the finite coarse-grained energy $E_N^{\rm cg}$ is defined as a minimization problem
\begin{equation}
\label{eq: finite CG without forces}
E_N^{\rm cg}(A,y):=\inf\limits_{\substack{u\in\R^N\\ u_1=y, u_N=N A}}\sum_{i=2}^N [W(u_i-u_{i-1})-W(A)].
\end{equation}
The main theorem of this section is the following.
\begin{theorem}
\label{theo: CG without forces}
\begin{enumerate}[(i)]
\item The coarse-grained energy, $E^{\rm cg}(A,y)=\lim\limits_{N\to\infty} E_N^{\rm cg}(A,y)$, exists and is given by
\begin{equation}
\label{eq: Ecg}
E^{\rm cg}(y)=W'(A)(A-y).
\end{equation}
In addition, for all $A,y\in \R$ we have $|E_N^{\rm cg}(A,y)-E^{\rm cg}(A,y)|\lesssim N^{-1}$.
\item The defect formation free energy $G_\infty(A)$ can be represented in terms of $E^{\rm cg}$ as
\begin{equation}
\label{eq: represent of Ginfy}
G_\infty(A)=-\log\frac{\int_{\R}\exp(-P(y)-\psi(y)-E^{\rm cg}(A,y))\,dy}{\int_{\R}\exp(-\psi(y)-E^{\rm cg}(A,y))\,dy}.
\end{equation}
\end{enumerate}
\end{theorem}
\begin{proof} We first prove \eqref{eq: Ecg}. The minimizer of the  minimization problem \eqref{eq: finite CG without forces}  satisfies the following Euler-Lagrange equation
\begin{equation*}
-W'(u_{i+1}-u_i)+W'(u_i-u_{i-1})=0,
\end{equation*}
which implies that $W'(u_i-u_{i-1})=\lambda$, i.e., $u_N-u_{N-1}=\ldots=u_2-u_1(=(W')^{-1}(\lambda))$. This implies that
\begin{equation*}
u_i-u_{i-1}=\frac{1}{N-1}\sum_{j=2}^N (u_j-u_{j-1})=\frac{NA-y}{N-1}=A+\frac{A-y}{N-1}.
\end{equation*}
Thus, we obtain
\begin{equation*}
E_N^{\rm cg}(A,y)=(N-1)\left[W\Big(A+\frac{A-y}{N-1}\Big)-W(A)\right].
\end{equation*}
By applying the mean value theorem twice, there exist $0\leq \theta,\theta'\leq 1$ such that
\begin{align*}
E_N^{\rm cg}(A,y)-E^{\rm cg}(A,y)&=(N-1)\left[W\Big(A+\frac{A-y}{N-1}\Big)-W(A)\right]-W'(A)(A-y)
\\&=(N-1)W'\Big(A+\theta\frac{A-y}{N-1}\Big)\frac{A-y}{N-1}-W'(A)(A-y)
\\&=\left[W'\Big(A+\theta\frac{A-y}{N-1}\Big)-W'(A)\right](A-y)
\\&=W''\Big(A+\theta'\frac{A-y}{N-1}\Big)\frac{(A-y)^2}{N-1}.
\end{align*}
Let $x\in\R$ and let $\sigma_x$ be the maximiser in the definition of $W(x)$. Then we have
\begin{equation*}
x= \Psi(\sigma_x)\quad\text{and}\quad W(x)=\sigma_x x-\log\int_{\R}\exp[-\psi(y)+\sigma_x y]\,dy.
\end{equation*}
It follows that
\begin{equation}
\label{eq: Legendre transform}
W'(x)=x\frac{d\sigma_x}{dx}+\sigma_x-\Psi(\sigma_x)\frac{d\sigma_x}{dx}=\sigma_x \quad \text{and}\quad W''(x)=\frac{d\sigma_x}{dx}=\frac{1}{\Psi'(\sigma_x)}.
\end{equation}
According to Lemma \ref{lem: bound of derivatives}, we have
\begin{equation*}
|W''(x)|\leq C
\end{equation*}
for all $x\in\R$. It implies that $\Big|W''\Big(A+\theta'\frac{A-y}{N-1}\Big)\Big|\leq C$ and hence,
\begin{equation*}
|E_N^{\rm cg}(A,y)-E^{\rm cg}(A,y)|\leq \frac{C(A-y)^2}{N-1},
\end{equation*}
which gives \eqref{eq: Ecg}. 

The representation \eqref{eq: represent of Ginfy} is a direct consequence of \eqref{eq: Ginf} and \eqref{eq: Ecg}. Indeed,
\begin{align*}
G_\infty(A)&\overset{\eqref{eq: Ginf}}{=}-\log\frac{\int_{\R} \exp[-(\psi+P)(y)+W'(A) y]\,dy}{\int_{\R} \exp[-\psi(y)+W'(A) y]\,dy}
\\&=-\log\frac{\int_{\R} \exp[-(\psi+P)(y)-W'(A)(A-y)]\,dy}{\int_{\R} \exp[-\psi(y)-W'(A)(A-y)]\,dy}
\\&\overset{\eqref{eq: Ecg}}{=}-\log\frac{\int_{\R}\exp(-P(y)-\psi(y)-E^{\rm cg}(A,y))\,dy}{\int_{\R}\exp(-\psi(y)-E^{\rm cg}(A,y))\,dy}.
\end{align*}
\end{proof}

\subsection{Propagation of error}
In this section, we prove Theorem \ref{th:intro:cg-approx} for the case without external forces. 
\begin{proof}[Proof of Theorem \ref{th:intro:cg-approx} for the case without external forces]\ \\

For shortening of the notation, we define $\tilde{\psi}:=\psi+P$.
We rewrite $G_N^{\rm cg}(A)$ as follows. 
\begin{align*}
G_N^{\rm cg}(A)&=-\log\frac{\int\exp[-\tilde{\psi}(y)-E_N^{\rm cg}(A,y)]\,dy}{\int\exp[-\psi(y)-E_N^{\rm cg}(A,y)]\,dy}
\\&=-\log \frac{\int\exp[-\tilde{\psi}(y)-E^{\rm cg}(A,y)]\,dy}{\int\exp[-\psi(y)-E^{\rm cg}(A,y)]\,dy}-\log\frac{\int\exp[-\tilde{\psi}(y)-E^{\rm cg}(A,y)-(E_N^{\rm cg}(A,y)-E^{\rm cg}(A,y))]\,dy}{\int\exp[-\tilde{\psi}(y)-E^{\rm cg}(A,y)]\,dy}
\\&\qquad+\log\frac{\int\exp[-\psi(y)-E^{\rm cg}(A,y)-(E_N^{\rm cg}(A,y)-E^{\rm cg}(A,y))]\,dy}{\int\exp[-\psi(y)-E^{\rm cg}(A,y)]\,dy}
\\&=-\log \frac{\int\exp[-\tilde{\psi}(y)-E^{\rm cg}(A,y)]\,dy}{\int\exp[-\psi(y)-E^{\rm cg}(A,y)]\,dy}-\log\Big\langle \exp[E_N^{\rm cg}(A,y)-E^{\rm cg}(A,y)]\Big\rangle_{\zeta_1}
\\&\qquad+\log\Big\langle \exp[E_N^{\rm cg}(A,y)-E^{\rm cg}(A,y)]\Big\rangle_{\zeta_2},
\end{align*}
where $\zeta_1$ and $\zeta_2$ are two probability measures defined by
\begin{align*}
\zeta_1(y)\,dy=\frac{\exp[-\tilde{\psi}(y)-E^{\rm cg}(y)]\,dy}{\int\exp[-\tilde{\psi}(y)-E^{\rm cg}(y)]\,dy}\quad\text{and}\quad\zeta_2(y)\,dy=\frac{\exp[-\psi(y)-E^{\rm cg}(y)]\,dy}{\int\exp[-\psi(y)-E^{\rm cg}(y)]\,dy}.
\end{align*}
We next show that the logarithmic terms are of order $O(N^{-1})$. The argument will be similar to the paragraph following \eqref{eq: temp estimate} in the proof of Proposition \ref{prop: limits of difference of G}. Applying the estimate $|e^t-1|\leq |t|e^{|t|}$ and using the estimate in Theorem \ref{theo: CG without forces}, we get
\begin{align*}
|\exp[E_N^{\rm cg}(A,y)-E^{\rm cg}(A,y)]-1|&\leq |E_N^{\rm cg}(A,y)-E^{\rm cg}(A,y)|\exp[|E_N^{\rm cg}(A,y)-E^{\rm cg}(A,y)|]
\\&\leq \frac{C}{N}(A-y)^2\exp\Big(\frac{C}{N}(A-y)^2\Big)
\\&\leq \frac{C}{N}(A-y)^2\exp\Big(\frac{\kappa_1+\varsigma_1}{2}(A-y)^2\Big), \quad\text{for $N$ sufficiently large}.
\end{align*}
Therefore,
\begin{align*}
\Big|\big\langle \exp[E_N^{\rm cg}(A,y)-E^{\rm cg}(A,y)]\big\rangle_{\zeta_1}-1\Big|&=\Big|\big\langle \exp[E_N^{\rm cg}(A,y)-E^{\rm cg}(A,y)]-1\big\rangle_{\zeta_1}\Big|
\\&\leq \Big\langle\big|\exp[E_N^{\rm cg}(A,y)-E^{\rm cg}(A,y)]-1\big|\Big\rangle_{\zeta_1}
\\&\leq \frac{C}{N}\Big\langle(A-y)^2\exp\Big(\frac{\kappa_1+\varsigma_1}{2}(A-y)^2\Big)\Big\rangle_{\zeta_1}.
\end{align*}
Thanks to Assumption \ref{assump on psi and P}, the last average term will be finite. Therefore,
\begin{equation*}
\Big|\big\langle \exp[E_N^{\rm cg}(A,y)-E^{\rm cg}(A,y)]\big\rangle_{\zeta_1}-1\Big|\leq  \frac{C}{N},
\end{equation*}
which implies that
\begin{equation*}
\left|\log\Big\langle \exp[E_N^{\rm cg}(A,y)-E^{\rm cg}(A,y)]\Big\rangle_{\zeta_1}\right|\leq\frac{C}{N}.
\end{equation*}
Similarly we also have
\begin{equation*}
\left|\log\Big\langle \exp[E_N^{\rm cg}(A,y)-E^{\rm cg}(A,y)]\Big\rangle_{\zeta_2}\right|\leq\frac{C}{N}.
\end{equation*}
Therefore, we obtain that
\begin{equation*}
\left|-\log\frac{\int\exp[-\tilde{\psi}(y)-E_N^{\rm cg}(y)]\,dy}{\int\exp[-\psi(y)-E_N^{\rm cg}(y)]\,dy}+\log \frac{\int\exp[-\tilde{\psi}(y)-E^{\rm cg}(y)]\,dy}{\int\exp[-\psi(y)-E^{\rm cg}(y)]\,dy}\right|\leq \frac{C}{N}.
\end{equation*}
This completes the proof.
\end{proof}
\section{External forces case}
\label{sec: forces}
In this section, we consider the case where the external forces are present. Recall that in this case, the perfect free energy is unchanged
\begin{equation}
\label{F_N(x,0) case 2}
\begin{cases}
F_N(A)=-\beta^{-1}\log \int_{\R^{N-1}}\exp\Big[-\beta\sum_{i=1}^N \psi(u_i-u_{i-1})\Big]\,du_1\ldots du_{N-1}\\
u_0=0, u_{N}=N A.
\end{cases}
\end{equation}
The deformed free energy is influenced by the external forces
\begin{equation}
\label{F_N(x,P) case 2}
\begin{cases}
F_N^P(A)=-\beta^{-1}\log \int_{\R^{N-1}}\exp\Big[-\beta\sum_{i=1}^N \psi_i(u_i-u_{i-1})-\beta P(u_1)\Big]\,du_1\ldots du_{N-1}\\
u_0=0, u_{N}=NA,
\end{cases}
\end{equation}
where $\psi_i(y)=\psi(y)+h_i y$. The defect-formation free energy is defined as the free energy difference, $G_\infty(A):=\lim_{N\to\infty} G_N(A)$, where
\begin{equation}
G_N(A)=F_N^P(A)-F_N(A).
\end{equation}
Finally, the finite-domain coarse-grained energy is given by
\begin{equation}
\label{eq: finite-domain CGenergy}
E_N^{\rm cg}(A,y):=\inf\limits_{\substack{u \in \mathbb{R}^N\\ u_1=y, u_N=N A}}\sum_{i=2}^N \Big[W(u_i-u_{i-1})-W(A)+h_i(u_i-u_{i-1})\Big].
\end{equation}
Recall also that the external forces $\{h_i\}_{i=1}^n$ satisfy Assumption \ref{assum: assumption on asssumption forces} and $H=\sum_{i=2}^\infty h_i$.
\subsection{Coarse-grained energy}
We now establish the formula for the coarse-grained energy, thus proving Theorem \ref{th:intro:cg} for the case with external forces.
\begin{theorem}
\label{theo: coarse-grained energy}
 The coarse-grained energy, $E^{\rm cg}(A,y):=\lim\limits_{N\to \infty} E_N^{\rm cg}(A,y)$, is given by
\begin{equation}
\label{eq: coarse-grained energy}
E^{\rm cg}(A,y)=(A-y)W'(A)+AH+\inf\limits_{\substack{v \in \mathbb{R}^N\\v_1=0}}J_\infty(A;v),
\end{equation}
where
\begin{equation}
\label{eq: Jinfty}
J_\infty(A;v)=\sum_{i=2}^\infty [W(A+v_i')-W(A)-W'(A)v_i'+h_iv_i'].
\end{equation}
In addition, for all $A,y\in\R$, we have the estimate
\begin{equation}
\label{eq: rate ENcg to Ecg}
|E_N^{\rm cg}(A,y)-E^{\rm cg}(A,y)|\lesssim  N^{-1}+A\,|\sum\limits_{i=N+1}^\infty h_i|+\sum\limits_{i=N+1}^\infty |h_i|^2.
\end{equation}
\end{theorem}
\begin{proof}
By changing variables $v_i'=u_i'-A$ and substituting to \eqref{eq: finite-domain CGenergy},  we obtain
\begin{equation}
E_N^{\rm cg}(A,y)=\inf\limits_{\substack{v \in \mathbb{R}^N\\v_1=y-A,v_N=0}} I_N(A;v),
\end{equation}
where 
\begin{align*}
I_N(A;v)&=\sum_{i=2}^N[W(A+v_i')-W(A)+h_i(v_i'+A)]
\\&=\sum_{i=2}^N[W(A+v_i')-W(A)-W'(A)v_i'+h_iv_i']+A\sum_{i=2}^N h_i +(A-y)W'(A)
\\&=J_N(A;v)+A\sum_{i=2}^N h_i +(A-y)W'(A),
\end{align*}
with
\begin{equation}
J_N(A;v)=\sum_{i=2}^N[W(A+v_i')-W(A)-W'(A)v_i'+h_iv_i'].
\end{equation}
Therefore
\begin{equation}
\label{pre-limit 1}
E_N^{\rm cg}(A,y)=A\sum_{i=2}^N h_i +(A-y)W'(A)+\inf\limits_{\substack{v\in\R^N\\v_1=y-A,v_N=0}} J_N(A;v).
\end{equation}
We now show that 
\begin{equation}
\label{pre-limit 1.2}
\lim\limits_{N\to\infty}\inf\limits_{\substack{v\in\R^N\\v_1=y-A,v_N=0}} J_N(A;v)=\inf\limits_{v_1=y-A}J_\infty(A;v),
\end{equation}
where
\begin{equation*}
J_\infty(A;v)=\sum_{i=2}^\infty [W(A+v_i')-W(A)-W'(A)v_i'+h_iv_i'].
\end{equation*}
In fact, since $J_\infty(A;v)$ depends only on $v_i'$, we have that
\begin{equation*}
\inf\limits_{\substack{v\in\R^N\\v_1=y-A}}J_\infty(A;v)=\inf\limits_{\substack{v\in\R^N\\v_1=0}}J_\infty(A;v).
\end{equation*}
To shorten the notation, we define $\Theta_i(A,z)=W(A+z)-W(A)-W'(A)z+h_iz$ so that
\begin{equation*}
J_N(A;v)=\sum_{i=2}^N\Theta(A,v_i'),\quad\text{and}\quad J_\infty(A;v)=\sum_{i=2}^\infty \Theta_i(A,v_i').
\end{equation*} 
A minimizer of $J_\infty(A;)$ satisfies the following Euler-Langrange equation for $i=2,\ldots, N$
\begin{equation*}
\Theta_i'(A;v_i')=0,
\end{equation*}
together with the boundary condition $v_1=y-A$. In particular, since $\Theta_i'(A,z)=W'(A+z)-W'(A)+h_i=W''(A+\theta z)z+h_i$ for some $\theta\in\R$, it follows that
\begin{equation*}
|v_i'|=\frac{|h_i|}{|W''(A+\theta v_i')|}\leq \frac{|h_i|}{\kappa_1}.
\end{equation*}
We define an admissible sequence $\tilde{v_i}$ as follows
\begin{equation*}
\tilde{v}_1=y-A,\quad \tilde{v}_N=0,\quad \tilde{v}'_i=v_i'+C_N,
\end{equation*}
for some $C_N$. Since $\{v_i'\}\in l^1$, we have $\sum_{i=2}^N v_i'\to a$  for some $a\in \mathbb{R}$. By summing up the above equalities, it follows that
\begin{equation*}
|C_N|\lesssim \frac{|y-A|+|a|}{N}.
\end{equation*}
Since $v_i'$ minimizes $\Theta_i$ we have
\begin{equation*}
0\leq \Theta_i(\tilde{v}_i')-\Theta_i(v_i)\lesssim C_N^2\lesssim N^{-2}.
\end{equation*}
As a consequence, we obtain
\begin{align}
\inf\limits_{\substack{w\in\R^N\\w_1=y-A,w_N=0}} J_N(A;w)&\leq J_N(A;\tilde{v})\nonumber
\\&=J_N(A;v)+\sum\limits_{i=2}^N[\Theta_i(\tilde{v}'_i)-\Theta_i(v_i')]\nonumber
\\&\leq J_\infty(A;v)+CN^{-1}+\sum\limits_{i=N+1}^\infty\Theta_i(v_i')\nonumber
\\&\leq J_\infty(A;v)+CN^{-1}+C\sum\limits_{i=N+1}^\infty |h_i|^2.\label{pre-limit 2}
\end{align}
Note that in the estimation above we have used the fact that $|\Theta_i(v_i')|\leq C(|h_i|^2+|v_i'|^2)|\leq C|h_i|^2$.

On the other hand, using again the fact that $v_i'$ minimizes $\Theta_i$ for each $i=2,\ldots,N$, we have
\begin{equation}
\label{pre-limit 3}
\inf\limits_{\substack{w\in\R^N\\w_1=y-A,w_N=0}} J_N(A;w)\geq J_N(A;v)=J_\infty(A;v)-\sum\limits_{i=N+1}^\infty \Theta(v_i')\geq J_\infty(A;v)-C\sum\limits_{i=N}^\infty |h_i|^2.
\end{equation}
From \eqref{pre-limit 2} and \eqref{pre-limit 3}, we obtain
\begin{equation}
\label{pre-limit 4}
\Big|\inf\limits_{\substack{v\in\R^N\\v_1=y-A,v_N=0}} J_N(A;v)-\inf\limits_{\substack{v\in\R^N\\v_1=y-A}}J_\infty(A;v)\Big|\lesssim N^{-1}+\sum\limits_{i=N+1}^\infty |h_i|^2,
\end{equation}
from which \eqref{pre-limit 1.2} follows. Finally, from \eqref{pre-limit 1} and \eqref{pre-limit 4}, we get
\begin{equation*}
|E_N^{\rm cg}(A,y)-\lim\limits_{N\to\infty}E_N^{\rm cg}(A,y)|\lesssim N^{-1}+A\,\bigg|\sum\limits_{i=N+1}^\infty h_i\bigg|+\sum\limits_{i=N+1}^\infty |h_i|^2,
\end{equation*}
which is \eqref{eq: rate ENcg to Ecg} (and hence \eqref{eq: coarse-grained energy}) as claimed. This finishes the proof of Theorem \ref{theo: coarse-grained energy}.
\end{proof}
\subsection{Thermodynamic limit}
The main result of this section is the following theorem on the  representation of the defect formation free energy.
\begin{theorem}
\label{theo: thermodynamic limit with forces}
The thermodynamic limit is given by
\begin{equation}
\label{eq: thermodynamic limit with forces}
G_\infty(A)=-\log\frac{\int_{\R}\exp[-(\psi_1+P)(y)-E^{\rm cg}(A,y)]\,dy}{\int_{\R}\exp[-\psi(y)-E^{\rm cg}_{\mathbf{h}=0}(A,y)]\,dy}.
\end{equation}
where $E^{\rm cg}(A,y)$ is defined in \eqref{eq: coarse-grained energy}. 
\end{theorem}
\begin{proof}[Proof of Theorem \ref{theo: thermodynamic limit with forces}]
The proof is analogous to that of Theorem \ref{theo: main theorem constrained case}  which consists of three main steps.
\begin{enumerate}[Step 1)]
\item Express the defect-formation free energy in terms of the energy difference and a ratio of the densities of random variables based on Lemma \ref{lem: aulem 1}.
\item Establish the limit of the energy difference.
\item Show that the ratio of the densities of random variables are of order $O(1/N)$.
\end{enumerate}
We now only sketch out the main computations in Step 1) and Step 2). Applying Lemma \ref{lem: aulem 1} for the case $\tilde{\psi}_1=\psi_1+P,\tilde{\psi}_2=\psi_i$, for $i=2,\ldots, N$ to obtain
\begin{align*}
W_N^P(A)&=\sup_{\sigma\in A}\Big\{\sigma A-\frac{1}{N}\int_{\R}\exp[-(\psi_1(y)+P(y)+\sigma y]\,dy-\frac{1}{N}\sum_{i=2}^N\exp[-\psi_i(y)+\sigma y]\,dy\Big\},
\\&=\sigma_N A-\frac{1}{N}\int_{\R}\exp[-(\psi_1(y)+P(y)+\sigma_N y]\,dy-\frac{1}{N}\sum_{i=2}^N\exp[-\psi_i(y)+\sigma_N y]\,dy.
\end{align*}
The optimal value $\sigma_N$ solves
\begin{align}
\label{eq: eqn for sigmaN case 2}
A&=\frac{1}{N}\frac{\int_{\R}y\exp[-(\psi_1(y)+P(y))+\sigma y]\,dy}{\int_{\R}\exp[-(\psi_1(y)+P(y))+\sigma y]\,dy}+\frac{1}{N}\sum_{i=2}^N\frac{\int_{\R}y\exp[-\psi_i(y)+\sigma y]\,dy}{\int_{\R}\exp[-\psi_i(y)+\sigma y]\,dy} \nonumber
\\&=\frac{1}{N}\Psi_P(\sigma-h_1)+\frac{1}{N}\sum_{i=2}^N\Psi(\sigma-h_i),
\end{align}
where $\Psi$ is defined in \eqref{eq: Psi} and $\Psi_P$ is given by
\begin{equation}
\Psi_P(\sigma)=\frac{\int_{\R}y\exp[-(\psi_1(y)+P(y))+\sigma y]\,dy}{\int_{\R}\exp[-(\psi_1(y)+P(y))+\sigma y]\,dy}.
\end{equation}
Since $W(A)$ is unchanged, it is the same as in \eqref{eq: W_N(A) 1}-\eqref{eq: W_N(A) 2}, so that
\begin{align}
N[W_N^P(A)-W(A)]&=N(\sigma_N-\sigma_0)A-\log\frac{\int_{\R}\exp[-(\psi_1(y)+P(y))+\sigma_N y]\,dy}{\int_{\R}\exp[-\psi(y)+\sigma_0 y]\,dy}\nonumber
\\&\qquad-\sum_{i=2}^N\log\frac{\int_{\R}\exp[-\psi_i(y)+\sigma_N y]\,dy}{\int_{\R}\exp[-\psi(y)+\sigma_0 y]\,dy}	\nonumber
\\&=N(\sigma_N-\sigma_0)A-\log\frac{\int_{\R}\exp[-(\psi_1(y)+P(y))+\sigma_N y]\,dy}{\int_{\R}\exp[-\psi(y)+\sigma_0 y]\,dy}\nonumber
\\&\qquad-\sum_{i=2}^N[W^*(-h_i+\sigma_N)-W^*(\sigma_0)],\label{eq: energy difference forces}
\end{align}
where $\sigma_0=W'(A)$.
We will need the following lemma whose proof is postponed after the proof of Theorem \ref{theo: thermodynamic limit with forces}.
\begin{lemma}
\label{lem: sigmaN-sigma0 case 2}
 It holds that
\begin{equation}
|\sigma_N-\sigma_0|\leq \frac{C}{N}.
\end{equation}
\end{lemma}

To proceed, we will compare this free energy difference with the finite-domain coarse-grained energy. Recalling that the latter is defined by (see \eqref{eq: finite-domain CGenergy}),
\begin{align}
\label{finite coarse-grained force type II}
E_N^{\rm cg}(y)&:=\inf\limits_{\substack{u:\{1,N\}\rightarrow \mathbb{R}\\ u(1)=y, u(N)=N A}}\sum_{i=2}^N \Big[W(u_i-u_{i-1})-W(A)+h_i(u_i-u_{i-1})\Big]\nonumber
\\&=A\sum\limits_{i=2}^N h_i+\inf\limits_{\substack{u:\{1,N\}\rightarrow \mathbb{R}\\ u(1)=y, u(N)=N A}}\sum_{i=2}^N \Big[W(u_i-u_{i-1})+h_i(u_i-u_{i-1})-(h_i A+W(A))\Big].
\end{align}
The Euler-Lagrange equation for a minimizer of $E_N^{\rm cg}$ is
\begin{equation*}
-W'(u_{i+1}-u_i)+W'(u_i-u_{i-1})-(h_{i+1}-h_i)=0,
\end{equation*}
which implies that
\begin{equation*}
W'(u_i-u_{i-1})=-h_i+\lambda
\end{equation*}
for $i=2,\ldots, N$ and for some $\lambda\in \R$.
We note that $(W')^{-1}(z)=(W^*)'(z)$, where $W^*$ is the Legendre transformation of $W$. It follows from the definition of $W$ that
\begin{equation*}
W^*(x)=\log\int \exp[-\psi(z)+x z]\,dz,
\end{equation*}
and so
\begin{equation*}
(W^*)'(x)=\frac{\int x \exp[-\psi(z)+x z]\,dz}{\int \exp[-\psi(z)+x z]\,dz}=\Psi(x).
\end{equation*}
Therefore, we obtain that
\begin{equation*}
u_i-u_{i-1}=(W')^{-1}(-h_i+\lambda)=(W^*)'(-h_i+\lambda)=\Psi(-h_i+\lambda).
\end{equation*}
Summing up these equalities from $i=2$ to $N$ and using the boundary condition on $u$, we obtain the following equation for $\lambda=\lambda_N$
\begin{equation}
\label{eq: eqn for lambda}
NA-y=\sum_{i=2}^N \Psi(-h_i+\lambda_N).
\end{equation}
Next, we use the following relations of the Legendre transform
\begin{equation*}
W(x)=W'(x)x-W^*(W'(x)), \qquad W'((W^*)'(x))=x
\end{equation*}
to obtain $W(A)=W'(A)A-W^*(W'(A))$ and
\begin{align*}
W(u_i-u_{i-1})&=W((W^*)'(-h_i+\lambda_N))
\\&=W'((W^*)'(-h_i+\lambda_N))(W^*)'(-h_i+\lambda_N)-W^*(W'((W^*)'(-h_i+\lambda_N)))
\\&=(-h_i+\lambda_N)(W^*)'(-h_i+\lambda_N)-W^*(-h_i+\lambda_N).
\end{align*}
Therefore, the sum inside the inf in \eqref{finite coarse-grained force type II} can be re-written as (recalling that $u_N=NA, u_1=y$)
\begin{align*}
&\sum\limits_{i=2}^N\Big[W(u_i-u_{i-1})+h_i(u_i-u_{i-1})-h_i A-W(A)\Big]
\\&\quad=\sum\limits_{i=2}^N\Big[(-h_i+\lambda_N)(W^*)'(-h_i+\lambda_N)-W^*(-h_i+\lambda_N)+h_i(W^*)'(-h_i+\lambda_N)
\\&\hspace*{3cm}-h_i A-W'(A)A+W^*(W'(A))\Big]
\\&\quad=\lambda_N\sum\limits_{i=2}^N(W^*)'(-h_i+\lambda_N)-\sum\limits_{i=2}^N\Big[W^*(-h_i+\lambda_N)-W^*(W'(A))+h_i A+W'(A)A\Big]
\\&\quad=\lambda_N\sum\limits_{i=2}^N(u_i-u_{i-1})-\sum\limits_{i=2}^N\Big[W^*(-h_i+\lambda_N)-W^*(W'(A))+h_i A+W'(A)A\Big]
\\&\quad=\lambda_N(NA-y)-A\sum_{i=1}^Nh_i-(N-1)W'(A)A-\sum\limits_{i=2}^N\Big[W^*(-h_i+\lambda_N)-W^*(W'(A))\Big].
\end{align*}
Substituting this expression back into \eqref{finite coarse-grained force type II}, we get
\begin{align}
E_N^{\rm cg}(y)&=\lambda_N(NA-y)-(N-1)W'(A)A-\sum_{i=2}^N[W^*(-h_i+\lambda_N)-W^*(W'(A))]\nonumber
\\&=\lambda_N(A-y)+(N-1)(\lambda_N-W'(A))A-\sum_{i=2}^N[W^*(-h_i+\lambda_N)-W^*(W'(A))].\label{eq: finite-domain CG 2}
\end{align}
It follows from \eqref{eq: energy difference forces} and \eqref{eq: finite-domain CG 2} that
\begin{align}
\label{eq: W_N-Ecg_N}
N[W_N(A)-W(A)]-E_N^{\rm cg}(A)&=(\sigma_N-\sigma_0)A-\log\frac{\int_{\R}\exp[-(\psi_1(y)+P(y))+\sigma_N y]\,dy}{\int_{\R}\exp[-\psi(y)+\sigma_0 y]\,dy}\nonumber
\\&\qquad+\sum_{i=2}^N\left([\sigma_N A-W^*(-h_i+\sigma_N)]-[\lambda_N A- W^*(-h_i+\lambda_N)]\right)\nonumber
\\&=(\sigma_N-\sigma_0)A-\log\frac{\int_{\R}\exp[-(\psi_1(y)+P(y))+\sigma_N y]\,dy}{\int_{\R}\exp[-\psi(y)+\sigma_0 y]\,dy}\nonumber
\\&\qquad+b_N(\sigma_N)-b_N(\lambda_N),
\end{align}
where
\begin{equation*}
b_{N}(x):=\sum_{i=2}^N[x A-W^*(-h_i+x)].
\end{equation*}
Then we have
\begin{align*}
&b_{N}'(x)=(N-1)A-\sum_{i=2}^N (W^*)'(-h_i+x),
\\&b_{N}''(x)=-\sum_{i=2}^N (W^*)''(-h_i+x)=-\sum_{i=2}^N \Psi'(-h_i+x)\leq 0,
\end{align*}
where we have used \eqref{eq: derivative of Psi} to obtain the last inequality.
Therefore $b_N'(x)$ is a non-increasing function. Furthermore, from \eqref{eq: eqn for lambda} and \eqref{eq: eqn for sigmaN case 2}, we have
\begin{align*}
&b_N'(\lambda_N)=(N-1)A-\sum_{i=2}^N (W^*)'(-h_i+\lambda_N)=y-A,\\
& b_N'(\sigma_N)=(N-1)A-\sum_{i=2}^N (W^*)'(-h_i+\sigma_N)=\Psi_P(\sigma_N-h_1)-A.
\end{align*}
Since $\frac{d}{d\sigma}\Psi_P(\sigma)\leq \frac{1}{\kappa_1+\varsigma_1}$, we have
\begin{align*}
|\Psi_P(\sigma_N-h_1)|&\leq |\Psi_P(0)|+\frac{1}{\kappa_1+\varsigma_1}|\sigma_N-h_1|\leq |\Psi_P(0)|+\frac{1}{\kappa_1+\varsigma_1}(|\sigma_0-h_1|+|\sigma_N-\sigma_0|)
\\&\leq\bigg(|\Psi_P(0)|+\frac{1}{\kappa_1+\varsigma_1}(|\sigma_0-h_1|+C)\bigg). 
\end{align*}
Therefore both $b_N'(\lambda_N)$ and $b_N'(\sigma_N)$ are uniformly bounded. It follows that
\begin{align*}
|b_N(\sigma_N)-b_N(\lambda_N)|&=|\sigma_N-\lambda_N| |b_N'(\theta_N)|
\\&\leq |\sigma_N-\lambda_N| \max\{|b_N'(\sigma_N),b_N'(\lambda_N)|\}
\\&\leq C |\sigma_N-\lambda_N|
\\&\leq C [|\sigma_N-W'(A)|+|\lambda_N-W'(A)|]
\\&\leq C (N-1)^{-1}.
\end{align*}
Substituting this estimate into \eqref{eq: W_N-Ecg_N}, we obtain
\begin{equation}
\label{eq: pre-theorem 1}
\left|N[W_N(A)-W(A)]-\bigg(E_N^{\rm cg}(A)+(\sigma_N-\sigma_0)A-\log\frac{\int_{\R}\exp[-(\psi_1(y)+P(y))+\sigma_N y]\,dy}{\int_{\R}\exp[-\psi(y)+\sigma_0 y]\,dy}\bigg)\right|\leq\frac{C}{N}.
\end{equation}
An analogous argument as in the proof of Proposition \ref{prop: limits of difference of G} we obtain 
\begin{equation}
\label{eq: pre-theorem 2}
\left|\log\frac{\int_{\R}\exp[-(\psi_1(y)+P(y))+\sigma_N y]\,dy}{\int_{\R}\exp[-\psi(y)+\sigma_0 y]\,dy}-\log\frac{\int_{\R}\exp[-(\psi_1(y)+P(y))+\sigma_0 y]\,dy}{\int_{\R}\exp[-\psi(y)+\sigma_0 y]\,dy}\right|\leq \frac{C}{N}.
\end{equation}
The assertion \eqref{eq: thermodynamic limit with forces} of Theorem \ref{theo: thermodynamic limit with forces} is then followed from \eqref{eq: pre-theorem 1}, Theorem \ref{theo: coarse-grained energy}, Lemma \ref{lem: sigmaN-sigma0 case 2} and \eqref{eq: pre-theorem 2}.
\end{proof}
We now prove Lemma \ref{lem: sigmaN-sigma0 case 2}.
\begin{proof}[Proof of Lemma \ref{lem: sigmaN-sigma0 case 2}]
Define $L(\sigma):=\frac{1}{N}\Psi_P(\sigma_N)+\frac{1}{N}\sum_{i=2}^N\Psi(\sigma_N-h_i)$. Then we have
\begin{equation*}
A=\Psi(\sigma_0)=L(\sigma_N).
\end{equation*}
Hence,
\begin{align*}
L(\sigma_N)-L(\sigma_0)=\Psi(\sigma_0)-L(\sigma_0)=\frac{1}{N}(\Psi(\sigma_0)-\Psi_P(\sigma_0-h_1))+\frac{1}{N}\sum_{i=2}^N(\Psi(\sigma_0)-\Psi(\sigma_0-h_i)).
\end{align*}
By the mean value theorem, there exists $\theta$ such that
\begin{equation}
L(\sigma_N)-L(\sigma_0)=L'(\theta)(\sigma_N-\sigma_0).
\end{equation}
We have
\begin{align*}
|L'(\theta)||\sigma_N-\sigma_0|=|L(\sigma_N)-L(\sigma_0)|&\leq \frac{1}{N}\left[|\Psi(\sigma_0)-\Psi_P(\sigma_0-h_1)|+\sum_{i=2}^N|\Psi(\sigma_0)-\Psi(\sigma_0-h_i)|\right]
\\&\leq\frac{1}{N}\left[|\Psi(\sigma_0)-\Psi_P(\sigma_0-h_1)|+\frac{1}{\kappa_1}\sum_{i=2}^N|h_i|\right].
\end{align*}
Since $0<|L'(\theta)|\leq C$, it implies that
\begin{equation*}
|\sigma_N-\sigma_0|\leq \frac{1}{N |L'(\theta)|}\left[|\Psi(\sigma_0)-\Psi_P(\sigma_0-h_1)|+\frac{1}{\kappa_1}\sum_{i=2}^N|h_i|\right]\leq \frac{C}{N}.
\end{equation*}
\end{proof}
\section{Harmonic potentials}
\label{sec: appendix}
In this section, we provide explicit computations for the quadratic case, 
\begin{equation}
\label{quadratic case}
\psi(y)=\alpha |y|^2, \qquad P(y)=\beta |y|^2,\quad \text{for some}~~\alpha,\beta>0.
\end{equation}

\subsection{Harmonic potentials without forcing}
\label{sec: hamonic no-force}
We recall that
\begin{align*}
F_N(A)=-\log\int_{\R^{N-1}}\exp\bigg[-\alpha\sum_{i=1}^{N-1}y_i^2-\alpha\Big(N A-\sum_{i=1}^{N-1}y_i\Big)^2\bigg]\,dy_1\ldots dy_{N-1}.
\end{align*}
and
\begin{align*}
F_N^P(A)=-\log\int_{\R^{N-1}}\exp\bigg[-(\alpha+\beta)\,y_1^2-\alpha\sum_{i=2}^{N-1}y_i^2-\alpha\Big(NA-\sum_{i=1}^{N-1}y_i\Big)^2\bigg]\,dy_1\ldots dy_{N-1}.
\end{align*}
The main result of the present section is the following.
\begin{theorem}
The defect-formation free energy is given by
\begin{align*}
G_N(A)&:=F^P_N(A)-F_N(A)\\
&=\frac{1}{2}\log\frac{\alpha+\beta}{\alpha}+\frac{\alpha\beta A^2}{\alpha+\beta}-\frac{N\alpha\beta^2 A^2}{(N(\alpha+\beta)-\beta)^2}+\frac{\alpha\beta A^2}{\alpha+\beta}\left(\frac{2\beta}{N(\alpha+\beta)-\beta}+\frac{\beta^2}{(N(\alpha+\beta)-\beta)^2}\right)\nonumber
\\&\qquad+\frac{1}{2}\log\Big(1-\frac{\beta}{N(\alpha+\beta)}\Big).
\end{align*}
The thermodynamic limit is given by
\begin{equation*}
G_\infty(A):=\lim_{N\rightarrow \infty}G_N(A)=\frac{\alpha\beta\,A^2}{\alpha+\beta}+\frac{1}{2}\log\frac{\alpha+\beta}{\alpha}.
\end{equation*}
Moreover, the following error estimate holds for all $A\in\R$ and $N\geq 2$ and for some positive constant $C$
\begin{equation*}
\left|G_N(A)-G_\infty(A)\right|\leq \frac{C}{N}.
\end{equation*}
\end{theorem}
\begin{proof}
The computations are lengthy but elementary. The following integrals will be used in the sequel
\begin{subequations}
\begin{align}
&\int_{\R}\exp(-a|y|^2+by)\,dy=\exp(\frac{b^2}{4a})\,\sqrt{\frac{\pi}{a}}\label{eq1},
\\&\int_{\R}y\exp(-a|y|^2+by)\,dy=\exp(\frac{b^2}{4a})\,\frac{b}{2a}\,\sqrt{\frac{\pi}{a}}\label{eq2},
\\&\int_{\R}y^2\exp(-a|y|^2+by)\,dy=\exp(\frac{b^2}{4a})\,\sqrt{\frac{\pi}{a}}\,\left(\frac{b^2}{4a^2}+\frac{1}{2a}\right)\label{eq3}.
\end{align}
\end{subequations}

From \eqref{sigma0}, we have
\begin{equation*}
A=\frac{\int_{\R} y\exp(-\alpha|y|^2+\sigma_0^* y)\,dy}{\int_{\R} \exp(-\alpha|y|^2+\sigma_0^* y)\,dy}=\frac{\sigma_0}{2\alpha},
\end{equation*}
which implies that
\begin{equation}
\label{sigma0 quadratic}
\sigma_0=2\alpha A.
\end{equation}
Similarly, from \eqref{sigmaP}, we have
\begin{align*}
A&=\frac{\int_{\R} y\exp(-\alpha |y|^2+\sigma_P y)\,dy}{\int_{\R} \exp(-\alpha |y|^2+\sigma_P y)\,dy}+\frac{1}{N}\left[\frac{\int_{\R} y\exp[-(\alpha+\beta)|y|^2+\sigma_P y]\,dy}{\int_{\R} \exp[-(\alpha+\beta)|y|^2+\sigma_P y)\,dy}-\frac{\int_{\R} y\exp(-\alpha |y|^2+\sigma_P y)\,dy}{\int_{\R} \exp(-\alpha |y|^2+\sigma_P y)\,dy}\right]
\\&=\frac{\sigma_P}{2\alpha}+\frac{1}{N}\left(\frac{\sigma_P}{2(\alpha+\beta)}-\frac{\sigma_P}{2\alpha}\right),
\end{align*}
which leads to
\begin{equation}
\label{sigmaP quadraic}
\sigma_P=2\alpha A\left(1+\frac{\beta}{\alpha+\beta-\frac{\beta}{N}}\frac{1}{N}\right).
\end{equation}
Therefore, we obtain
\begin{align*}
N[W_N^P(A)-W(A)]&=N(\sigma_P-\sigma_0)x-N\log\frac{\int_{\R}\exp[-\psi(y)+\sigma_P y]\,dy}{\int_{\R}\exp[\psi(y)+\sigma_0y]\,dy}
\\&\qquad -\log\frac{\int_{\R}\exp(-(\psi+P)(y)+\sigma_P y)\,dy}{\int_{\R}\exp[\psi(y)+\sigma_P y]\,dy}
\\&=N\frac{2\alpha\beta A}{\alpha+\beta-\beta/N}\frac{1}{N}A-N\log\frac{\exp\big[\frac{\sigma_P^2}{4\alpha}\big]\sqrt{\frac{\pi}{\alpha}}}{\exp\big[\frac{\sigma_0^2}{4\alpha}\big]\sqrt{\frac{\pi}{\alpha}}}
\\&\qquad -\log\frac{\exp\big[\frac{\sigma_P^2}{4(\alpha+\beta)}\big]\sqrt{\frac{\pi}{\alpha+\beta}}}{\exp\big[\frac{\sigma_P^2}{4\alpha}\big]\sqrt{\frac{\pi}{\alpha}}}
\\&=\frac{1}{2}\log\frac{\alpha+\beta}{\alpha}+\frac{2\alpha\beta A^2}{\alpha+\beta-\beta/N}-N\frac{\sigma_P^2-\sigma_0^2}{4\alpha}-\left(\frac{\sigma_P^2}{4(\alpha+\beta)}-\frac{\sigma_P^2}{4\alpha}\right)
\\&=\frac{1}{2}\log\frac{\alpha+\beta}{\alpha}+\frac{2\alpha\beta A^2}{\alpha+\beta-\beta/N}-N\frac{4\alpha^2A^2\big(1+\frac{\beta}{N(\alpha+\beta)-\beta}\big)^2-4\alpha^2A^2}{4\alpha}
\\&\qquad-4\alpha^2A^2\left(1+\frac{\beta}{N(\alpha+\beta)-\beta}\right)^2\left(\frac{1}{4(\alpha+\beta)}-\frac{1}{4\alpha}\right)
\\&=\frac{1}{2}\log\frac{\alpha+\beta}{\alpha}+\frac{\alpha\beta A^2}{\alpha+\beta}-\frac{N\alpha\beta^2 A^2}{(N(\alpha+\beta)-\beta)^2}+\frac{\alpha\beta A^2}{\alpha+\beta}\left(\frac{2\beta}{N(\alpha+\beta)-\beta}+\frac{\beta^2}{(N(\alpha+\beta)-\beta)^2}\right)
\\& =\frac{1}{2}\log\frac{\alpha+\beta}{\alpha}+\frac{\alpha\beta A^2}{\alpha+\beta}+O(1/N).
\end{align*}
As a consequence, taking the limit $N\rightarrow \infty$, we achieve
\begin{equation*}
\lim_{N\rightarrow\infty}N[W_N^P(A)-W(A)]=\frac{\alpha\beta\,A^2}{\alpha+\beta}+\log\frac{\alpha+\beta}{\alpha}.
\end{equation*}
We next compute $g_{N,A}(0)$ and $g^P_{N,A}(0)$ using the formula \eqref{f_Nx as Fourier trans }.
\begin{equation}
\label{g_Nx(0)}
g_{N,A}(0)=\frac{1}{2\pi}\int_{\R}\prod_{j=1}^N\int_{\R}\exp\Big(i(y_j-m_j)\frac{1}{\sqrt{N}}\xi\Big)\mu_j^{\sigma_0}(dy_j)\, d\xi,
\end{equation}
where
\[
\mu_j^{\sigma_0}(dy_j)=Z^{-1}\exp(-\alpha\,y_j^2+\sigma_0\,y_j)\,dy_j.
\]
Hence
\begin{equation*}
\int_{\R}\exp\Big(i(y_j-m_j)\frac{1}{\sqrt{N}}\xi\Big)\mu_j^{\sigma_0}(dy_j)=Z^{-1}\int_{\R}e^{-i\frac{1}{\sqrt{N}}\,m_j\,\xi}e^{-\alpha\,y_j^2+(\sigma_0+\frac{i\,\xi}{\sqrt{N}})y_j}\,dy_j
\end{equation*}
According to \eqref{eq1}-\eqref{eq3}, we have
\begin{align*}
&Z=\int_{\R}\exp(-\alpha\,y_j^2+\sigma_0\,y_j)\,dy_j=\exp\big[\frac{\sigma_0^2}{4\alpha}\big]\sqrt{\frac{\pi}{\alpha}},
\\&\int_{\R}e^{-i\frac{1}{\sqrt{N}}\,m_j\,\xi}e^{-\alpha\,y_j^2+(\sigma_0+\frac{i\,\xi}{\sqrt{N}})y_j}\,dy_j=\exp\Big(\frac{(\sigma_0+\frac{i\,\xi}{\sqrt{N}})^2}{4\alpha}\Big)\sqrt{\frac{\pi}{\alpha}}.
\end{align*}
Therefore
\begin{align*}
\int_{\R}\exp\Big(i(y_j-m_j)\frac{1}{\sqrt{N}}\xi\Big)\mu_j^{\sigma_0}(dy_j)&=\exp\left[-\frac{i\,m_j\,\xi}{\sqrt{N}}+\frac{1}{4\alpha}\Big(\big(\sigma_0+\frac{i\,\xi}{\sqrt{N}}\big)^2-\sigma_0^2\Big)\right]
\\&=\exp\left[\frac{i\,\xi}{\sqrt{N}}\big(\frac{\sigma_0}{2\alpha}-m_j\big)-\frac{\xi^2}{4\alpha\, N}\right].
\end{align*}
Since $m_j=\frac{\sigma_0}{2\alpha}$, it follows that
\begin{align*}
\prod_{j=1}^N\int_{\R}\exp\Big(i(y_j-m_j)\frac{1}{\sqrt{N}}\xi\Big)\mu_j^{\sigma_0}(dy_j)&=\prod_{j=1}^N\exp\left[\frac{i\,\xi}{\sqrt{N}}\big(\frac{\sigma_0}{2\alpha}-m_j\big)-\frac{\xi^2}{4\alpha\,N}\right]
\\&=\exp\left[\frac{i\,\xi}{\sqrt{N}}\big(N\frac{\sigma_0}{2\alpha}-\sum_{j=1}^N m_j\big)-\frac{\xi^2}{4\alpha}\right]=\exp(-\frac{\xi^2}{4\alpha}).
\end{align*}
Substituting back to \eqref{g_Nx(0)} we obtain
\begin{equation}
\label{g_Nx(0) final}
g_{N,A}(0)=\frac{1}{2\pi}\int_{\R}\exp(-\frac{\xi^2}{4\alpha})\,d\xi=\frac{1}{2\pi} \times 2\sqrt{\pi\,\alpha}.
\end{equation}
Next we compute $g^P_{N,A}(0)$ using
\begin{equation}
g^P_{N,A}(0)=\frac{1}{2\pi}\int_{\R}\exp\Big(i(y_1-m_{P,1})\frac{1}{\sqrt{N}}\xi\Big)\nu^{\sigma_P}(dy_1)\prod_{j=2}^N\int_{\R}\exp\Big(i(y_j-m_{P,j})\frac{1}{\sqrt{N}}\xi\Big)\mu_j^{\sigma_P}(dy_j)\, d\xi,
\end{equation}
where
\begin{align}
&\nu^{\sigma_P}(dy_1)=Z_{P,1}^{-1}\exp(-(\alpha+\beta)\,y_1^2+\sigma_P\,y_1)\,dy_1,
\\&\mu_j^{\sigma_P}(dy_j)=Z_{P,i}^{-1}\exp(-\alpha y_i^2+\sigma_P)\,dy_i
\end{align}
From \eqref{eq1}, the normalising constants $Z_{P,i}$ are given by
\begin{align*}
&Z_{P,1}=\int_{\R}\exp(-(\alpha+\beta)\,y_1^2+\sigma_P\,y_1)\,dy_1=\exp\Big(\frac{\sigma_P^2}{4(\alpha+\beta)}\Big)\sqrt{\frac{\pi}{\alpha+\beta}},
\\&Z_{P,i}=\int_{\R}\exp(-\alpha y_i^2+\sigma_P)\,dy_i=\exp\left(\frac{\sigma_P^2}{4\alpha}\right)\sqrt{\frac{\pi}{\alpha}}.
\end{align*}
Similarly as above, we find
\begin{align*}
\int_{\R}\exp\Big(i(y_1-m_{P,1})\frac{1}{\sqrt{N}}\xi\Big)\nu^{\sigma_P}(dy_1)&=\exp\left[\frac{i\,\xi}{\sqrt{N}}\big(\frac{\sigma_P}{2(\alpha+\beta)}-m_{P,1}\big)-\frac{\xi^2}{4(\alpha+\beta)\,N}\right]
\\&=\exp\left[-\frac{\xi^2}{4(\alpha+\beta)\,N}\right]
\end{align*}
since
\[
m_{P,1}=\int_{\R}y_1\,\nu^{\sigma_P}(dy_1)=\frac{\sigma_P}{2(\alpha+\beta)}.
\]
Additionally
\begin{align*}
\prod_{j=2}^N\int_{\R}\exp\Big(i(y_j-m_{P,j})\frac{1}{\sqrt{N}}\xi\Big)\mu_j^{\sigma_P}(dy_j)=\exp(-\frac{N-1}{N}\frac{\xi^2}{4\alpha}).
\end{align*}
Therefore, we obtain
\begin{align}
\label{g_NxP final}
g^P_{N,A}(0)&=\frac{1}{2\pi}\int_{\R}\exp\left[-\frac{\xi^2}{4(\alpha+\beta)\,N}-(N-1)\frac{\xi^2}{4\alpha\,N}\right]\,d\xi=\frac{1}{2\pi}\int_{\R}\exp\left[-\frac{\xi^2}{4}\times\frac{1-\frac{\beta}{N(\alpha+\beta)}}{\alpha}\right]\,d\xi\nonumber
\\&=\frac{1}{2\pi}\times 2\sqrt{\pi}\times\sqrt{\frac{\alpha}{1-\frac{\beta}{N(\alpha+\beta)}}}.
\end{align}
From \eqref{g_Nx(0) final} and \eqref{g_NxP final}, we get
\begin{equation}
\log\frac{g^P_{N,A}(0)}{g_{N,A}(0)}=\log\sqrt{\frac{1}{1-\frac{\beta}{N(\alpha+\beta)}}}=\frac{1}{2}\log \frac{1}{1-\frac{\beta}{N(\alpha+\beta)}}.
\end{equation}
Since $0\leq \frac{1}{2}\log\frac{1}{1-z}\leq z$ for $0\leq z\leq \frac{1}{2}$, we have
\[
0\leq \log\frac{g^P_{N,A}(0)}{g_{N,A}(0)}\leq \frac{1}{N}\frac{\beta}{\alpha+\beta}\qquad\text{for}\quad N\geq 2.
\]
\end{proof}

\subsection{Harmonic potentials with external forces}
Now we consider the quadratic case with external forces. Recall that the perfect energy is
\begin{equation}
\label{quadraic F_N(x,0) case 2}
\begin{cases}
F_N(A)=-\beta^{-1}\log \int_{\R^{N-1}}\exp\Big[-\beta\sum_{i=1}^N \psi(u_i-u_{i-1})\Big]\,du_1\ldots du_{N-1}\\
u_0=0, u_{N}=N A.
\end{cases}
\end{equation}
and the deformed energy is
\begin{equation}
\label{quadratic F_N(x,P) case 2}
\begin{cases}
F_N^P(A)=-\beta^{-1}\log \int_{\R^{N-1}}\exp\Big[-\beta\sum_{i=1}^N \psi_i(u_i-u_{i-1})-\beta P(u_1)\Big]\,du_1\ldots du_{N-1}\\
u_0=0, u_{N}=NA,
\end{cases}
\end{equation}
where $\psi_i(y)=\psi(y)+h_i y=\alpha y^2+h_i y$, where $\{h_i\}$ represent the external forces.

In view of Assumption \ref{assum: assumption on asssumption forces} we define
\begin{equation*}
H:=\sum_{i=2}^\infty h_i\quad\text{and}\quad \bar{H}=\sum_{i=2}^\infty h_i^2.
\end{equation*}
The main result of this section is the following.
\begin{theorem} The thermodynamic limit has the following explicit formula 
\begin{equation*}
G_\infty(A)=\frac{1}{2}\log\frac{\alpha+\beta}{\alpha}+\frac{\alpha\beta A^2}{\alpha+\beta}+\frac{\alpha A h_1}{\alpha+\beta}-\frac{h_1^2}{4(\alpha+\beta)}+AH-\frac{1}{4\alpha}\bar{H}.
\end{equation*}
\end{theorem}
\begin{proof}
In this case
\begin{align*}
W(A)&=\sup_{\sigma}\{\sigma A-\log\int_{\R}\exp(-\psi(y)+\sigma y)\,dy\}
\\&=\sigma_0 A-\log\int_{\R}\exp(-\psi(y)+\sigma_0 y)\,dy,
\end{align*}
where $\sigma_0=2\alpha A$,  which is obtained similarly as in the case without forces. And,
\begin{align*}
W_N(A)&=\sup_{\sigma\in A}\Big\{\sigma A-\frac{1}{N}\int_{\R}\exp[-(\psi_1(y)+P(y)+\sigma y]\,dy-\frac{1}{N}\sum_{i=2}^N\exp[-\psi_i(y)+\sigma y]\,dy\Big\},
\\&=\sigma_N A-\frac{1}{N}\int_{\R}\exp[-(\psi_1(y)+P(y)+\sigma_N y]\,dy-\frac{1}{N}\sum_{i=2}^N\exp[-\psi_i(y)+\sigma_N y]\,dy
\end{align*}
where $\sigma_N$ solves
\begin{align*}
A&=\frac{1}{N}\frac{\int_{\R}y\exp[-(\psi_1(y)+P(y))+\sigma y]\,dy}{\int_{\R}\exp[-(\psi_1(y)+P(y))+\sigma y]\,dy}+\frac{1}{N}\sum_{i=2}^N\frac{\int_{\R}y\exp[-\psi_i(y)+\sigma y]\,dy}{\int_{\R}\exp[-\psi_i(y)+\sigma y]\,dy}
\\&=\frac{1}{N}\left[\frac{\sigma-h_1}{2(\alpha+\beta)}+\sum_{i=2}^N\frac{\sigma-h_i}{2\alpha}\right],
\end{align*}
which results in
\begin{equation}
\sigma_N=\frac{1}{1-\frac{\beta}{(\alpha+\beta)N}}\left(2\alpha A+\frac{1}{N}\sum_{i=1}^Nh_i-\frac{\beta h_1}{N(\alpha+\beta)}\right).
\end{equation}
Next we compute
\begin{align}
N[W_N(A)-W(A)]&=N(\sigma_N-\sigma_0)A-\log\frac{\int_{\R}\exp[-(\psi_1(y)+P(y))+\sigma_N y]\,dy}{\int_{\R}\exp[-\psi(y)+\sigma_0 y]\,dy}-\sum_{i=2}^N\log\frac{\int_{\R}\exp[-\psi_i(y)+\sigma_N y]\,dy}{\int_{\R}\exp[-\psi(y)+\sigma_0 y]\,dy}	
\\&=(I)+(II)+(III).
\end{align}
The first term:
\begin{align*}
(I)&=NA\left[\frac{1}{1-\frac{\beta}{(\alpha+\beta)N}}\left(2\alpha A+\frac{1}{N}\sum_{i=1}^Nh_i-\frac{\beta h_1}{N(\alpha+\beta)}\right)-2\alpha A\right]
\\&=\frac{1}{1-\frac{\beta}{(\alpha+\beta)N}}\left[ \frac{2\alpha\beta A^2}{\alpha+\beta}+A\left(\sum_{i=1}^N h_i-\frac{\beta h_1}{\alpha+\beta}\right)\right].
\end{align*}
The second term:
\begin{align*}
(II)&=-\log\frac{\exp\Big[\frac{(\sigma_N-h_1)^2}{4(\alpha+\beta)}\Big]\sqrt{\frac{\pi}{\alpha+\beta}}}{\exp\Big[\frac{\sigma_0^2}{4\alpha}\Big]\sqrt{\frac{\pi}{\alpha}}}	
\\&=\frac{1}{2}\log\frac{\alpha+\beta}{\alpha}-\left[\frac{(\sigma_N-h_1)^2}{4(\alpha+\beta)}-\frac{\sigma_0^2}{4\alpha}\right]
\\&=\frac{1}{2}\log\frac{\alpha+\beta}{\alpha}-\left[\frac{1}{\left(1-\frac{\beta}{(\alpha+\beta)N}\right)^2}\frac{1}{4(\alpha+\beta)}\left(2\alpha A-h_1+\frac{1}{N}\sum_{i=1}^Nh_i\right)^2-\frac{1}{4\alpha}4\alpha^2 A^2\right]
\\&=-\frac{1}{\left(1-\frac{\beta}{(\alpha+\beta)N}\right)^2}\left[\frac{\alpha^2 A^2}{\alpha+\beta}-\left(1-\frac{\beta}{(\alpha+\beta)N}\right)^2\alpha A^2-\frac{\alpha A h_1}{\alpha+\beta}+\frac{h_1^2}{4(\alpha+\beta)}\right]
\\&\qquad-\frac{1}{\left(1-\frac{\beta}{(\alpha+\beta)N}\right)^2}\left[\frac{1}{4(\alpha+\beta)}\left(\frac{2}{N}(2\alpha A-h_1)\sum_{i=1}^Nh_i+\frac{1}{N^2}(\sum_{i=1}^Nh_i)^2\right) \right]
\\&\qquad+\frac{1}{2}\log\frac{\alpha+\beta}{\alpha}.
\end{align*}
The third term:
\begin{align*}
(III)&=-\sum_{i=2}^N\log\frac{\exp\Big[\frac{(\sigma_N-h_i)^2}{4\alpha}\Big]\sqrt{\frac{\pi}{\alpha}}}{\exp\Big[\frac{\sigma_0^2}{4\alpha}\Big]\sqrt{\frac{\pi}{\alpha}}}
\\&=-\sum_{i=2}^N\left(\frac{(\sigma_N-h_i)^2}{4\alpha}-\frac{\sigma_0^2}{4\alpha}\right)
\\&=-\frac{1}{4\alpha}\sum_{i=2}^N(\sigma_N^2-\sigma_0^2-2h_i\sigma_N+h_i^2)
\\&=-\frac{1}{4\alpha}\left[(N-1)(\sigma_N^2-\sigma_0^2)-2\sigma_N\sum_{i=2}^Nh_i+\sum_{i=2}^N h_i^2\right]
\\&=-\frac{1}{4\alpha}\left[(N-1)(2\alpha A)^2\left(\frac{1}{\left(1-\frac{\beta}{(\alpha+\beta)N}\right)^2}-1\right)-2\sigma_N\sum_{i=2}^Nh_i+\sum_{i=2}^N h_i^2\right]
\\&\qquad -\frac{N-1}{4\alpha}\frac{1}{\left(1-\frac{\beta}{(\alpha+\beta)N}\right)^2}\left[\Big(\frac{1}{N}\sum_{i=1}^N h_i-\frac{\beta h_1}{N(\alpha+\beta)}\Big)^2+4\alpha A \Big(\frac{1}{N}\sum_{i=1}^N h_i-\frac{\beta h_1}{N(\alpha+\beta)}\Big)\right]
\\&=-\frac{1}{4\alpha}\left[\frac{1}{1-\frac{\beta}{(\alpha+\beta)N}}\left(\frac{1}{1-\frac{\beta}{(\alpha+\beta)N}}+1\right)\frac{4\alpha^2\beta A^2}{\alpha+\beta}\frac{N-1}{N}-2\sigma_N\sum_{i=2}^Nh_i+\sum_{i=2}^N h_i^2\right]
\\&\qquad -\frac{N-1}{N}A\left(\sum_{i=1}^Nh_i-\frac{\beta h_1}{\alpha+\beta}\right)
\\&-\frac{1}{4\alpha}\frac{N-1}{N^2}\frac{1}{\left(1-\frac{\beta}{(\alpha+\beta)N}\right)^2}\Big(\sum_{i=1}^N h_i-\frac{\beta h_1}{(\alpha+\beta)}\Big)^2.
\end{align*}
Bring all three terms together we obtain
\begin{align*}
&N[W_N(A)-W(A)]
\\&=\frac{1}{2}\log\frac{\alpha+\beta}{\alpha}+\frac{1}{1-\frac{\beta}{(\alpha+\beta)N}}\left[ \frac{2\alpha\beta A^2}{\alpha+\beta}+A\left(\sum_{i=1}^N h_i-\frac{\beta h_1}{\alpha+\beta}\right)\right]
\\&\qquad-\frac{1}{\left(1-\frac{\beta}{(\alpha+\beta)N}\right)^2}\left[\frac{\alpha^2 A^2}{\alpha+\beta}-\left(1-\frac{\beta}{(\alpha+\beta)N}\right)^2\alpha A^2-\frac{\alpha A h_1}{\alpha+\beta}+\frac{h_1^2}{4(\alpha+\beta)}\right]
\\&\qquad-\frac{1}{4\alpha}\left[\frac{1}{1-\frac{\beta}{(\alpha+\beta)N}}\left(\frac{1}{1-\frac{\beta}{(\alpha+\beta)N}}+1\right)\frac{4\alpha^2\beta A^2}{\alpha+\beta}\frac{N-1}{N}-2\sigma_N\sum_{i=2}^Nh_i+\sum_{i=2}^N h_i^2\right]
\\&\qquad-\frac{N-1}{N}A\left(\sum_{i=1}^Nh_i-\frac{\beta h_1}{\alpha+\beta}\right)
\\&\qquad-\frac{1}{\left(1-\frac{\beta}{(\alpha+\beta)N}\right)^2}\left[\frac{1}{4(\alpha+\beta)}\left(\frac{2}{N}(2\alpha A-h_1)\sum_{i=1}^Nh_i+\frac{1}{N^2}(\sum_{i=1}^Nh_i)^2\right) \right]
\\&\qquad -\frac{1}{4\alpha}\frac{N-1}{N^2}\frac{1}{\left(1-\frac{\beta}{(\alpha+\beta)N}\right)^2}\Big(\sum_{i=1}^N h_i-\frac{\beta h_1}{(\alpha+\beta)}\Big)^2.
\end{align*}
Taking the limit $N\to\infty$, we get
\begin{align*}
\lim_{N\rightarrow \infty}[N(W_N(A)-W(A))]=\frac{1}{2}\log\frac{\alpha+\beta}{\alpha}+\frac{\alpha\beta A^2}{\alpha+\beta}+\frac{\alpha A h_1}{\alpha+\beta}-\frac{h_1^2}{4(\alpha+\beta)}+AH-\frac{1}{4\alpha}\bar{H}.
\end{align*}

The ratio $\frac{g^P_{N,A}(0)}{g_{N,A}(0)}$ is the same as in Section \ref{sec: hamonic no-force}. The assertion of the theorem is then followed from the above limit.
\end{proof}
\subsection{Finite coarse-grained energy and representation of the thermodynamic limit}
In the quadratic case, $\psi(y)=\alpha y^2$, then
\begin{equation}
W(y)=\sigma_0 y-\log\int \exp[-\psi(z)+\sigma_0 z]\,dz = \alpha y^2+\frac{1}{2}(\log\alpha-\log\pi),
\end{equation}
and $W^*(y)=\frac{1}{4\alpha}y^2+\frac{1}{2}(\log\pi-\log\alpha)$.
In this case, $\lambda$ satisfies
\begin{equation*}
u_i-u_{i-1}=\frac{1}{2\alpha}(-h_i +\lambda).
\end{equation*}
We obtain $\lambda=\frac{2\alpha(NA-y)}{N-1}+\frac{1}{N-1}\sum\limits_{i=2}^N h_i$. Therefore,
\begin{align*}
E_N^{\rm cg}(y)&=\sum\limits_{i=2}^N\left[\frac{1}{4\alpha}(-h_i+\lambda)^2+\frac{1}{2\alpha}h_i(-h_i+\lambda)-\alpha A^2\right]
\\&=\sum\limits_{i=2}^N\left[-\frac{1}{4\alpha}h_i^2+\frac{1}{4\alpha} \lambda^2-\alpha A^2\right]
\\&=(N-1)\frac{1}{4\alpha}(\lambda^2-4\alpha^2 A^2)-\frac{1}{4\alpha}\sum\limits_{i=2}^Nh_i^2
\\&=(N-1)\frac{1}{4\alpha}\left[\left(2\alpha A+\frac{2\alpha}{N-1}(A-y)+\frac{1}{N-1}\sum\limits_{i=2}^N h_i\right)^2-4\alpha^2 A^2\right]-\frac{1}{4\alpha}\sum\limits_{i=2}^Nh_i^2
\\&=(N-1)\frac{1}{4\alpha}\left[4\alpha A\left(\frac{2\alpha}{N-1}(A-y)+\frac{1}{N-1}\sum\limits_{i=2}^N h_i\right)+\left(\frac{2\alpha}{N-1}(A-y)+\frac{1}{N-1}\sum\limits_{i=2}^N h_i\right)^2\right]
\\&\qquad-\frac{1}{4\alpha}\sum\limits_{i=2}^Nh_i^2
\\&=2\alpha A(A-y)+A\sum_{i=2}^Nh_i-\frac{1}{4\alpha}\sum\limits_{i=2}^Nh_i^2+\frac{1}{4\alpha(N-1)}\left(2\alpha(A-y)+\sum\limits_{i=2}^N h_i\right)^2
\end{align*}
Taking the limit $N\to \infty$, we obtain
\begin{equation}
\lim_{N\to \infty} E_N^{\rm cg}(y)=2\alpha A(A-y)+AH-\frac{1}{4\alpha}\bar{H}=E^{\rm cg}(y)
\end{equation}

Further more
\begin{equation*}
|E_N^{\rm cg}(y)-E^{\rm cg}(y)|\leq \frac{1}{N-1}(c_1 y^2+c_2y +c_3).
\end{equation*}

The thermodynamic limit can be represented as
\begin{align*}
G_\infty&=\frac{1}{2}\log\frac{\alpha+\beta}{\alpha}+\frac{\alpha\beta A^2}{\alpha+\beta}+\frac{\alpha A h_1}{\alpha+\beta}-\frac{h_1^2}{4(\alpha+\beta)}+AH-\frac{1}{4\alpha}\bar{H}
\\&=-\log \frac{\int \exp[-(\psi(y)+P(y)+h_1 y)-E^{\rm cg}(y)]\,dy}{\int \exp[-\psi(y)-E_{\mathbf{h}=0}^{\rm cg}(y)]\,dy},
\end{align*}
which is in accordance with the general result in Section \ref{sec: forces}.
\subsection{Harmonic coarse-graining}
In this section, we provide a direct method to coarse-graining for the harmonic case. We consider as before the potential energy
\begin{equation}
V(u) = \sum_{i=1}^N \psi(u_i - u_{i-1}),
\end{equation}
and the perturbed energy
\begin{equation}
    V(u) + P(u_1) = \sum_{i=1}^N \psi(u_i - u_{i-1}) + P(u_1),
\end{equation}
where we consider the harmonic case $\psi(r) = K_1 r^2$ and $P(r) = K_2 r^2.$
We are interested in the free energy difference
\begin{equation}
    F_N(x,P) - F_N(x,0) = - \log \int_{\R^{N-1}} \exp( - V(u) - P(u)) + \log \int_{\R^{N-1}} \exp( - V(u) ). 
\end{equation}
As seen above, this can be analytically computed.  However, we consider
coarse-graining the potential energy and using the free energy difference of
the coarse-grained model to 
approximate the free energy difference for the full model.  We show that
the free energy difference for the coarsened model is identical to that
of the full model.

Since our interactions are first-neighbor only and the defect potential is 
restricted to the first bond, we leave the first bond fully resolved and
use a uniform coarsening elsewhere.  That is, associated to the displacement
$w \in \R^M,$ we have the piecewise linear interpolation operator 
$I_h : \R^M \rightarrow \R^N$ where $(I_h w)_{p (j-1) + 1} = u_j.$  
In particular, $N = p (M-1) + 1.$
The coarse-grained potential energy is then
\begin{equation}
    V_{\rm cg}(w) 
		= \psi(w_1 - w_0) + \sum_{i=2}^M p \psi(p^{-1} ( w_j - w_{j-1})) 
		= K_1 w_1^2 + \sum_{i=2}^M K_1 p^{-1} (w_j - w_{j-1})^2.
\end{equation}
The technique given here for computing the free energy will differ from that in 
the main working note.  Here, we successively complete squares on the energy, starting from  
$w_{M-1},$ and we define a recurrence for the coefficients $c_i, d_i,$ and $f_i$ that are
introduced in the expansion.  
\begin{equation*}
\begin{split}
V_{\rm cg}(w) 
  &= K_1 w_1^2 + \sum_{i=2}^N K_1 p^{-1} (w_j - w_{j-1})^2 \\
  &=  K_1 p^{-1} N^2 x^2 - 2 K_1 p^{-1} N x w_{M-1}+ \sum_{i=2}^{M-1} K_1 p^{-1}  \left[ 2 w_j - 2 w_j w_{j-1} \right]
	    + K_1 (1+ p^{-1} ) w_1^2 \\
  &= K_1 p^{-1} \left[ N^2 x^2 + 2 (w_{N-1} - \frac{1}{2} (N x + w_{N-2}))^2 - \frac{1}{2} (N x + w_{N-2})^2 \right. \\
	& \quad \left. +  \sum_{i=2}^{M-2} (  2 w_j - 2 w_j w_{j-1} )  \right]
	    + K_1 (1+ p^{-1} ) w_1^2 	 \\
  &= K_1 p^{-1} \left[ 
	f_i N^2 x^2 + \sum_{i=m}^{M-1} c_i (w_i - c_i^{-1} (w_{i-1} + d_i N x))^2 
	- c_{m}^{-1} (w_{m-1} + d_{m} N x)^2  \right. \\
	& \quad \left. +  \sum_{i=2}^{m-1} (  2 w_j - 2 w_j w_{j-1} )  \right]
	    + K_1 (1+ p^{-1} ) w_1^2 	
\end{split}	
\end{equation*}
where the coefficients satisfy the following recurrences:
\begin{align*}
c_{i-1} &= 2 - c_i^{-1} \qquad c_{M-1} = 2 \\
d_{i-1} &= \frac{d_i}{c_i} \qquad d_{M-1} = 1 \\
f_{i-1} &= f_i - \frac{d_i^2}{c_i} \qquad f_{M-1} = 1
\end{align*}
we then find for $i = 2,\dots,M-1,$
\begin{align*}
c_{i} &= \frac{M - i + 1}{M - i} \\
d_{i} &= \frac{1}{M-i} \\
f_{i} &= \frac{1}{M-i}
\end{align*}

So, for the coarse-grained energy, we compute:
\begin{equation*}
\begin{split}
V_{\rm cg}(w) 
  &= K_1 p^{-1} \left[ 
	\sum_{i=2}^{M-1} c_i (w_i - c_i^{-1} (w_{i-1} + d_i N x))^2 
	 \right. \\
	& \quad \left. +  \left(\frac{M}{M-1} + p - 1  \right) \left(w_1 -  \frac{d_{1} N x}{c_{1} + p -1}   \right)^2  + \frac{N^2 x^2 p}{M + (p-1)(M-1)}\right] 
\end{split}	
\end{equation*}
where the lowest order terms do not satisfy the recursion because of the factor of $p,$ but they 
are computed manually.  Using the same recursion, we can also transform the energy with the 
defect, taking care to modify the lowest term.  
\begin{equation*}
\begin{split}
V_{\rm cg}(w) + P(w_1)
  &= K_1 p^{-1} \left[ 
	\sum_{i=2}^{M-1} c_i (w_i - c_i^{-1} (w_{i-1} + d_i N x))^2 
	 \right] \\
	& \quad  +  \left( \frac{K_1}{p} \left(\frac{M}{M-1} + p - 1  \right) + K_2 \right) \left(w_1 -  \frac{K_1 d_{1} N x}{K_1 (c_{1} + p -1) + K_2 p}   \right)^2 \\
	& \quad + \frac{K_1 N^2 x^2 (K_1 + K_2) }{K_1( M + (p-1)(M-1)) + K_2 p (M-1)}  
\end{split}	
\end{equation*}
When we take free energy differences, we can directly integrate starting from $w_{M-1}$ downwards, and the 
only differences in the two energies are in the lowest terms.  Also, we note that 
$M + (p-1)(M-1) = N,$ and $p(M-1) = N-1,$ so that the $p$ will fall out.  
We have 
\begin{equation*}
F_M^{\rm cg}(x, P) -F_M^{\rm cg}(x, 0)
= \frac{K_1 N^2 x^2 (K_1 + K_2) }{K_1 N + K_2 (N-1)}   -
\frac{N^2 x^2}{N} + \frac{1}{2} \log \left[ \frac{K_1 N + K_2 (N-1)}{ K_1 N} \right]
\end{equation*}
We note that this is exactly the result arrived at in Section \ref{sec: hamonic no-force}, and that 
there is no $p$ or $M$ dependence here.  That is, any uniform coarse-graining of the chain that leaves the first
bond refined exactly computes the free energy difference.

\section{Numerical Free Energy}
\label{sec: numerical}
We present numerical experiments to illustrate the results of the paper using standard free energy computation techniques as in~\cite{lrs}.  We compare the finite chain energy $G_N,$ coarse grained energy $G^{\rm cg}_N,$ and $G_\infty$ computed using numerical quadrature of the limit expression.  We see the theoretically expected $N^{-1}$ rate of convergence, where the asymptotic rate is observed to be valid even for small $N,$ and we numerically demonstrate that
$G_N - G^{\rm cg}_N$ also seems to decay as $N^{-1}.$   

\subsection{Free Energy Perturbation}
A standard approach for computing free energy differences is called
the free energy perturbation technique which rewrites the free energy
difference as an ensemble average of the energy perturbation 
with respect to the invariant measure of the unperturbed system.   To compute the free energy difference 
between $V$ and $V^P,$  we write 
\begin{align*}
G_N = F^P_N - F_N 
&= - \log \frac{\int_\Gamma \exp(- V^P(z)) dz}{\int_\Gamma \exp(- V(z)) dz}  \\
&= - \log \frac{\int_\Gamma \exp(- (V^P(z)- V(z))) \exp( - V(z)) dz}{\int_\Gamma \exp(- V(z)) dz}  \\
&= - \log \langle \exp(- P(u) ) \rangle_{\mu_0} 
\end{align*}
Therefore, one samples $\exp(- P(u) )$ with respect to the invariant
measure given by $V.$  The last step uses the assumption of a separable Hamiltonian.

\subsection{Staging}
Direct sampling to compute the free energy perturbation can be very slow to converge
when $V^P - V$ is large, particularly when the minima of $V$ and $V^P$ are separated.   
Many samples are chosen near the global minimum of $V,$ which may not
significantly contribute to the value of the integral.  
Instead, one can employ staging, where 
the free energy difference is broken into a telescopic sum.   That is, we write
$V_\lambda = V + \lambda P,$ and $F_\lambda = - \beta^{-1} \log \int_\Gamma V_\lambda(z) \, dz.$
Then the free energy difference can be written
\begin{equation*}
 F_N^P - F_N = \sum_{i=1}^{N_{\rm stages}} F_{\lambda_i} - F_{\lambda_{i-1}},
\end{equation*}
so that one must sample $\exp( - \beta (\lambda_i - \lambda_{i-1}) P)$ with respect to the invariant
measure corresponding to $V_{\lambda_{i-1}}$.  Since the energies $V_{\lambda_i}$ and $V_{\lambda_{i-1}}$ are closer
than $V$ and $V^P,$ it can convergence and reduce the overall computed variance.

\subsection{Metropolis Adjusted Langevin Algorithm}
In the following, we apply the Metropolis Adjusted Langevin Algorithm (MALA),
which proceeds as a series of overdamped Langevin steps followed by an
accept/reject step:
\begin{equation*}
q^* = q^n - h \nabla V(q^n) + \sqrt{ h } G  \qquad \text{where }
G \sim \mathcal{N}(0,Id)
\end{equation*}
Then we accept the new step and set $q^{n+1} = q^*$ with probability
\begin{equation*}
r(q^n, q^*) =  \min \left(1,   \frac{T(q^*, dq^n) \mu(d q^*)}{T(q^n, dq^*) \mu(d q^n)} \right) 
\end{equation*}
where 
\begin{equation*}
T(q,dq') =  \left( \frac{1}{4 \pi h} \right)^{d/2}  \exp\left( \frac{-  | q' - q + h \nabla V|^2}{4 h} \right).
\end{equation*}
Otherwise, we set $q^{n+1} = q^n.$  The accept/reject step assures that we
are sampling the invariant measure $\mu{d q}$ for any stepsize $h.$
The choice of $h$ is driven by two competing interests: larger $h$
speeds up convergence from the initial condition to the invariant measure,
whereas smaller $h$ means that a step is more likely to be accepted.

\subsection{Unforced Nonlinear Chain}

We consider the nonlinear energy
\begin{equation}
\label{eq:nonlinear_psi}
\psi(r) = \frac{1}{2} (r-1)^4 + \frac{1}{2} r^2,
\end{equation}
which satisfies the growth assumptions~\eqref{ass: assumption} and was also the test case used 
in~\cite{BlancBrisLegollPatz2010}.
We take a harmonic defect perturbation $P(y) = y^2$ and choose $A=2.$
\begin{figure}
\centerline{ \includegraphics[width=3in]{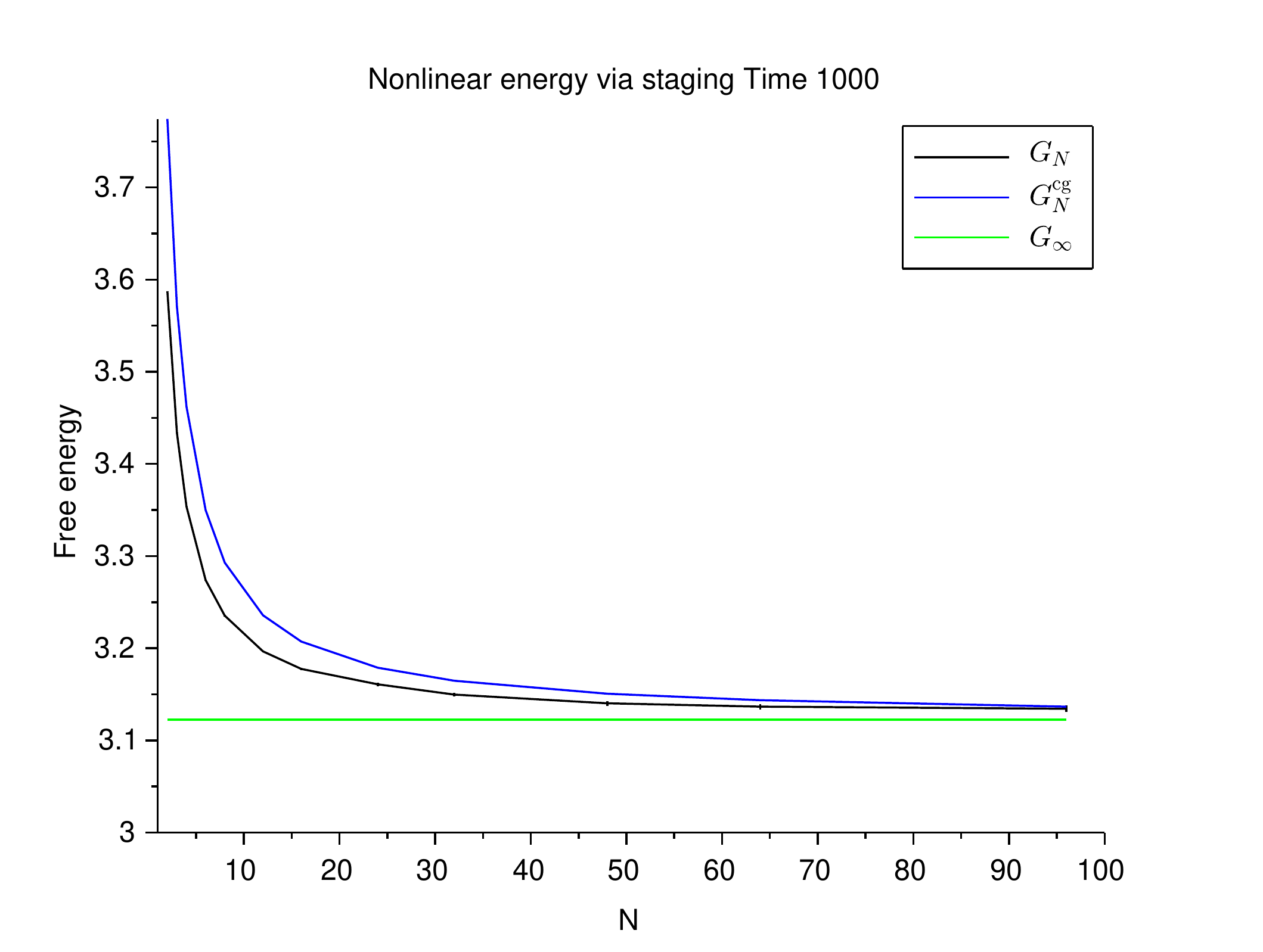}
\includegraphics[width=3in]{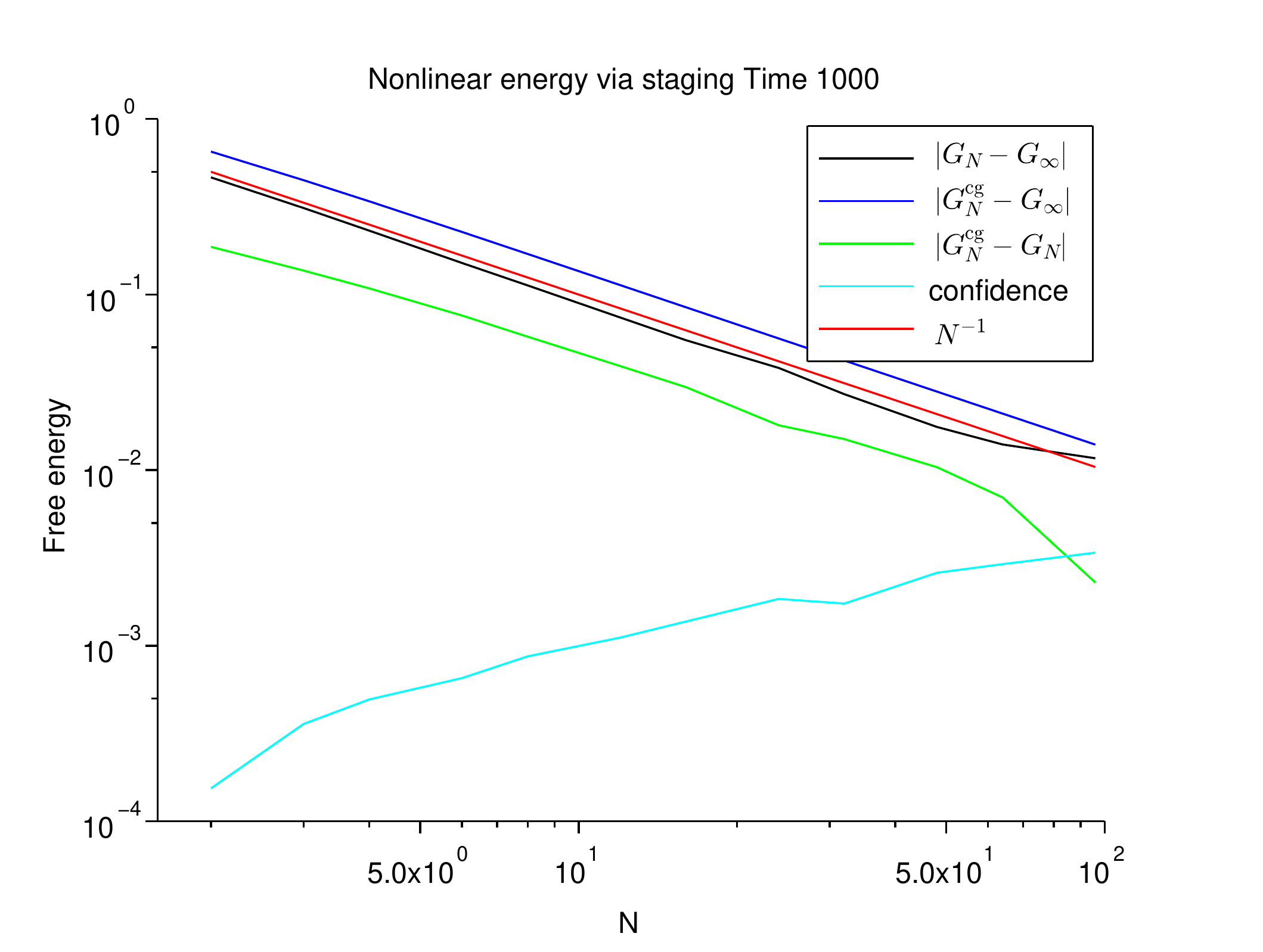}
}
\caption{\label{fig:free_en_nonlinear}For the nonlinear potential, the 
free energy difference is sampled using staging, and the result is compared to the coarse-grained
approximation.  The limiting free energy is computed via numerical quadrature and plotted in green.  On the right, we show the rate of convergence to the limiting energy $G_\infty,$ where both approximations show O($N^{-1}$) convergence.  The difference between $G_N$ and $G_N^{\rm cg}$
is also O($N^{-1}$). 
}
\end{figure}

The free difference $G_N$ is sampled using the MALA algorithm with 100 staging steps and 100 independent replicas
to compute confidence intervals.  In addition,
the coarse-grained approximation $G_N^{\rm cg}$ is also computed.  Due to the 1D nature of the problem, the minimizer for the CG energy is
given by an affine function, so that the computations involved are low-dimensional integrals.  First the energy density 
\begin{equation*}
W(A)=\sup_{\sigma\in \R}\Big\{\sigma A -  \log \int_{\R}\exp(-\psi(y)+\sigma y)\,dy\Big\}
\end{equation*}
is computed by quadrature, giving coarse-grained energy
\begin{equation*}
E_N^{\rm cg}(y)=(N-1)\left[W\Big(A+\frac{A-y}{N-1}\Big)-W(A)\right].
\end{equation*}
Then we may compute
\begin{equation*}
G_N^{\rm cg}=-\log \frac{\int \exp(-P(y)-\psi(y)-E_N^{\rm cg}(y)\,dy}{\int \exp(-\psi(y)-E_N^{\rm cg}(y))\,dy}
\end{equation*}
using standard quadrature techniques.  In Figure~\ref{fig:free_en_nonlinear}, the sampled free energy
difference $G_N$ is compared to $G_N^{\rm cg}$ as well as $G_\infty.$  The O($N^{-1}$) convergence
is seen throughout the chosen range of $N.$  We observe through the numerics that $|G_N - G_N^{\rm cg}|$ 
is also O($N^{-1}$). At present, this can not be explained by our theory.

\subsection{External Forces}

As a second example, we compute the free energy difference with external forces but no
defect potential.  Using different decay rates for the external forces provides an analog 
for the slow decay in the elastic field that surrounds defects in higher dimensional problems.
The non-defective chain has nonlinear interaction potentials~\eqref{eq:nonlinear_psi}, and the defective
chain has external forces $f_i = i^{-p}$ on each degree of freedom $u_i,$ or 
$h_i = - \sum_{j=i}^{N-1} f_j.$   The free energy $G_N$ chain is sampled using MALA with 100 stages,
and the limiting expression for $G_\infty$~\eqref{eq: thermodynamic limit with forces} is computed 
numerically, where it is noted that the minimization problem in the limit separates into single
variable problems.  As the forces decay sufficiently fast, a Taylor series approximation is used
for all but the first four terms in $E^{cg}(A).$
\begin{figure}
\centerline{ \includegraphics[width=3in]{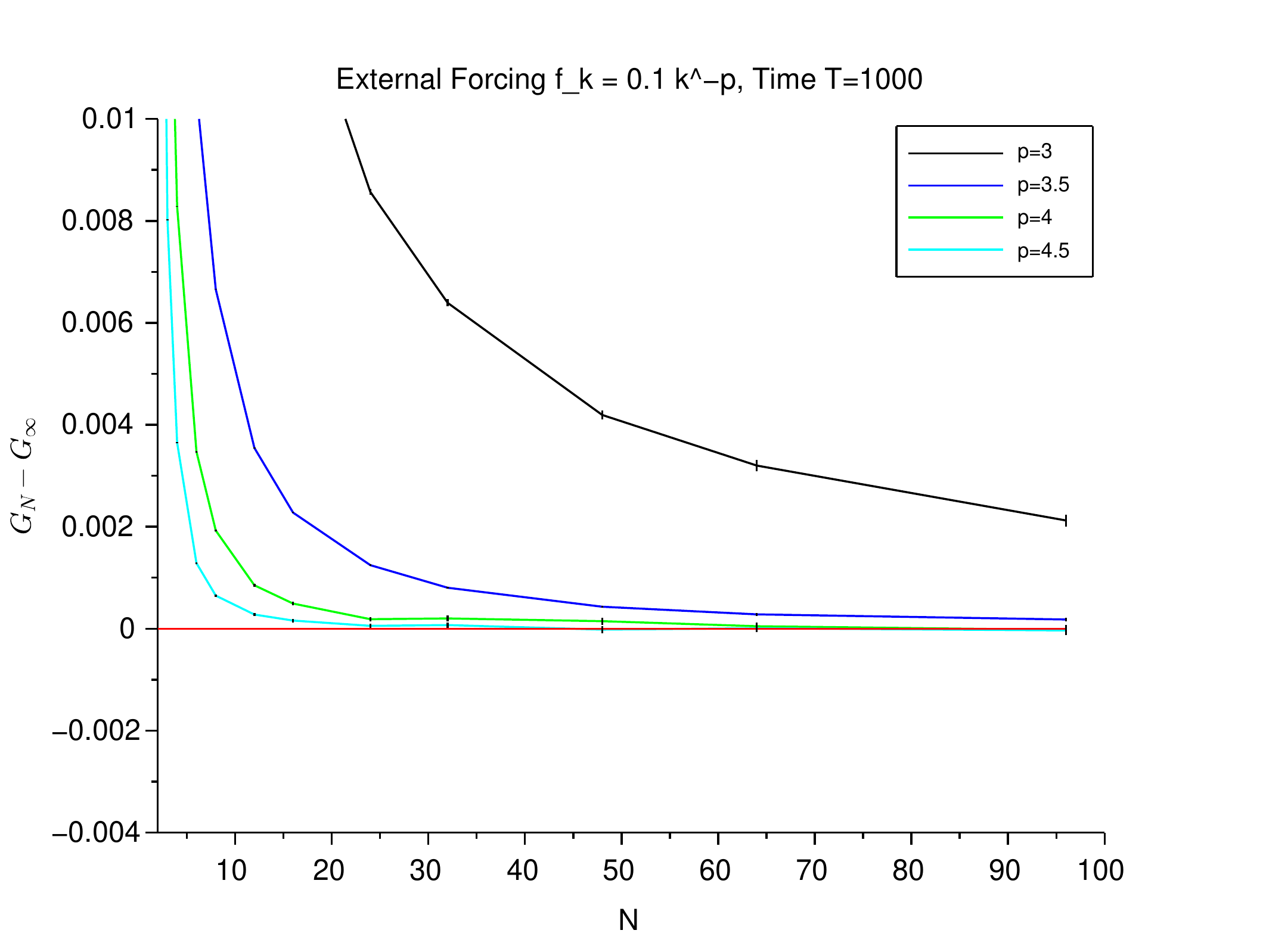}
\includegraphics[width=3in]{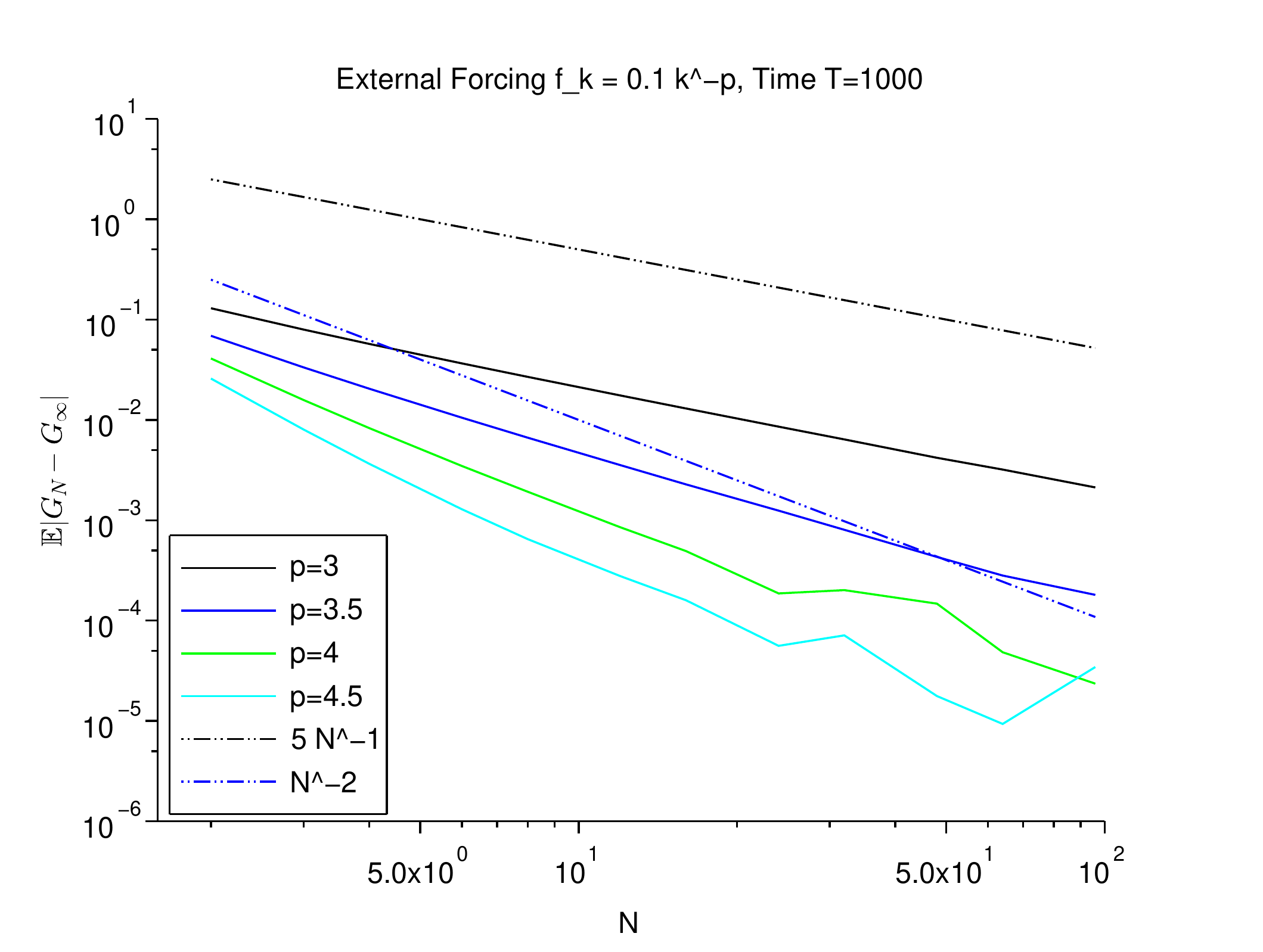}
}
\caption{\label{fig:free_en_forced} A nonlinear chain is sampled where the defect is
modeled by decaying forces $f_i = i^{-p}.$  The 
free energy difference is sampled using staging for varying $N$, and the limiting free energy is computed as in~\eqref{eq: thermodynamic limit with forces}.  On the right, we show the rate of convergence to the limiting energy $G_\infty,$ where the approximations seem to have $p$-dependent rates of
convergence.  Note that for exponents $p=4$ and $p=4.5$, the computed energy quickly approaches
the limiting energy up to statistical noise.
}
\end{figure}

In figure~\ref{fig:free_en_forced}, the differences $G_N - G_\infty$ are plotted for various rates of decay 
in the external forces $f_i = i^{-p}, \quad p = 3, 3.5, 4, 4.5.$  The observed rates of convergence depend
on the decay rate and are observed to be faster than O($N^{-1}$).  

\section{Conclusion}

We have provided a rigorous analysis of the defect-formation free energy~\eqref{eq: def of GN} 
for a one-dimensional, nearest neighbour chain with nonlinear local defect and external forces.  
The limiting energy is written in terms of a coarse-grained energy that is based on the Cauchy-Born
strain energy density.  The form of the coarse-grained energy was chosen because its variational
structure is amenable to analysis and approximation by methods in variational mechanics.

The analysis required many restrictions on the model.  The nonlinear perturbation $P$ could be 
extended to a finite region rather than the first bond without additional difficulty.  Including
interactions beyond nearest neighbour in $V$ would entail extension of the arguments here, for example 
the bonds are no longer independently distributed in Lemma~\ref{lem: aulem 1},
compare the work done for the free energy density in~\cite{BlancBrisLegollPatz2010}.  Moving beyond
one spatial dimension for the chain requires significant additional work; however, the inclusion
of external forces was motivated in part by the higher dimensional cases as a way to model
slowly-decaying stress field around a defect present in dimensions higher than
one.
\\ \ \\
\noindent\textbf{Acknowledgements:} MHD and CO were supported by ERC Starting Grant 335120.

\newcommand{\etalchar}[1]{$^{#1}$}

\end{document}